\documentclass[a4paper,11pt]{article}         
\usepackage{mathptmx}      % use Times fonts if available

\usepackage{amsmath}  
\usepackage{amsfonts} 
\usepackage{amsthm}
\usepackage{graphicx}
\usepackage{array} % Extension of the tabular environment.
\usepackage{enumerate} % Consider use of package `paralist'.
\usepackage[latin1]{inputenc} %Input encoding.
\usepackage[T1]{fontenc} %Output encoding.  
\usepackage{url}
\usepackage[usenames]{xcolor}    
\usepackage{framed}\colorlet{shadecolor}{black!25!white} % Provisional for frames
\usepackage{tikz}
\usepackage{multirow} % For the tables
\usepackage{comment} % environment {comment}
\usepackage{subcaption}
\usepackage{fancyhdr}
\fancyheadoffset{30pt} %Fa la capçalera més llarga que l'amplada del text.
\fancyfootoffset{30pt} %Fa el peu més llarg que l'amplada del text.

\newcommand{\Reals}{\mathbb{R}}  
\newcommand{\R}{\mathbb{R}}

\newcommand{\E}{\operatorname{E}}
\newcommand{\erfc}{\operatorname{erfc}}
\newcommand{\erf}{\operatorname{erf}}
\newcommand{\ind}{\mathbf{1}}
\newtheorem{thm}{Theorem}[section]

\newtheorem{teo}{Teorema}[section]
\newtheorem{prop}[teo]{Proposition}
\newtheorem{example}[thm]{Example}

  \newcommand{\condprob}[2]{\raise2pt
    \hbox{%
      \mathsurround=0pt$#1$}
    \ \raise-1pt\hbox{\scalebox{1.2}[1.5]{/}}\,% 
    \raise-2pt
    \hbox{%
      \mathsurround=0pt$#2$}
  }
\newlength\savedwidth
   
\renewcommand{\tablename}{\small Table}

\begin{document}  
\title{{On the minimum of a conditioned Brownian bridge}
%\thanks{Grants or other notes}
}
\author{Aureli Alabert \\ %\at
           Department of Mathematics \\
           Universitat Aut\`onoma de Barcelona \\  
           08193 Bellaterra, Catalonia \\  
           \url{Aureli.Alabert@uab.cat}  
           \and
           Ricard Caballero \\ %\at
           Department of Mathematics \\
           Universitat Aut\`onoma de Barcelona \\  
           08193 Bellaterra, Catalonia \\  
           \url{rcaballero@mat.uab.cat}
}

\date{\today}
\thispagestyle{empty}
\maketitle

\begin{abstract}  
 We study the law of the minimum of a Brownian bridge, conditioned to take
 specific values at specific points, and the law of the location of the minimum. 
 They are used to compare some non-adaptive optimisation algorithms for black-box
 functions for which the Brownian bridge is an appropriate probabilistic model 
 and only a few points can be sampled.
\end{abstract}

\section{Introduction}

We study the law of the minimum of a Brownian bridge conditioned to pass through given points
in the interval $[0,1]$, and the location of this minimum. Our motivation is the investigation 
of the performance of algorithms based on 
probabilistic models in expensive black-box optimisation. 

The probabilistic model point of view assumes the existence of a probability space from where the function
at hand has been drawn. The choice of points to sample is guided by the probabilistic properties
of this random function. Eventually, the values of the function at the points 
already sampled can be used to
decide the next sampling point (\emph{adaptive algorithms}) or can be neglected (\emph{non-adaptive} or 
\emph{passive algorithms}).

We assume here that the probabilistic model is completely specified, and given by the standard Brownian bridge
on the interval $[0,1]$;
that means, the function to be optimised is a path of a standard Brownian motion process, 
conditioned to take certain values $x_0$ at $t=0$ and $x_1$ at $t=1$.
More generally, one could set up a \emph{statistical model} (a family of probabilistic models depending on some parameters)
and improve sequentially the knowledge of the parameters using the values observed while sampling.

Probabilistic models try to account for heavy multimodality in the 
objective function. 
The irregularity and the 
independence of values over disjoint intervals of the Brownian bridge
and other Markovian stochastic processes represent well this multimodality, although at a 
very local scale the functions found in practice are usually smooth.

Our main interest is in expensive black-box functions from which only a few points
can be sampled, where it is more important to have an estimation of the absolute error incurred 
in approximating the true minimum
than the convergence, the speed of convergence, or the complexity properties of the algorithm.

In this paper we establish some facts about the law of the minimum of a Brownian bridge on the interval $[0,1]$, 
conditioned to hit some points $(t_i,x_i)\in(0,1)\times\Reals$. The density function of the law can be
computed exactly, but we argue that it is better to use simulation to obtain its features. 
We then use these simulations to evaluate empirically the performance of three non-adaptive algorithms 
when only small samples are allowed. Two of them are very simple and known: pure random sampling
and sampling at equidistant points. We propose a third one, which performs better in the present setting.
New adaptive algorithms in the same setting will be presented and compared elsewhere. 

The Brownian bridge model in optimisation has been studied by several authors, from the point of 
view of the asymptotic properties of the algorithms (see, e.g. Locatelli \cite{MR1432103}, 
Ritter \cite{MR1085383}, Calvin \cite{MR1325821, MR1459267, MR2086951}).
We mention here just two facts: 
\begin{enumerate}
  \item
\emph{Long-run performance:} Sampling at $n$ equidistant points and taking the value of the best 
sampled point as the approximation of the true minimum has an absolute error whose expectation 
is $O(1/\sqrt{n})$.
The best adaptive algorithm is better than the best non-adaptive algorithm concerning improvement rates, 
but asymptotically both are $O(1/\sqrt{n})$. Thus, sampling at equidistant points is
optimal in the long run.   
  \item
\emph{Complexity:} For algorithms using $n$ function evaluations, the convergence to zero 
of the mean error cannot be $O(e^{-cn})$ for any constant $c$. (This convergence order is indeed attained in 
unimodal functions, for example by Fibonacci search.)
\end{enumerate}
A general survey of probabilistic methods for optimisation can be found in Zhigljavsky and {\v{Z}}ilinskas \cite[ch.~4]{MR2361744}.

We establish some notations and preliminaries in Section \ref{sec:prelim}. Section \ref{sec:Pi} is devoted to 
computing the probability that the minimum lies in a given interval determined by two of the 
conditioning points. In Section \ref{sec:simmin} we show how to simulate the law of the
global minimum of the process. In Section \ref{sec:naopt}
we test and compare the three non-adaptive algorithms from the point of view of the
expected difference between the best sampled point and the true minimum of the path, 
when the evaluation points are few; we also present an empirical sensitivity analysis 
when the underlying model is an Ornstein--Uhlenbeck bridge instead of a Brownian bridge. 
Finally, in Section \ref{sec:SimLoc}, we 
compute the conditional distribution of the location of the minimum of a single Brownian bridge 
given the value of this minimum, and we show how to use it to simulate the location of the minimum
of the whole process.

\section{Preliminaries}\label{sec:prelim}

In the sequel, for a given stochastic process 
$Z:=\{Z_t,\ t\in I\}$, defined on a closed interval $I\subset\R$, 
we denote by 
\begin{equation*}
 m(Z):=\min_{t\in I}Z_t
 \qquad \text{and}\qquad
 \theta(Z):=\arg \min_{t \in I}Z_t 
\end{equation*}
the random variables giving the minimum value of $Z$ and its location, respectively. In the cases we will treat
here, the minimum exists and is unique with probability 1 but, to avoid any ambiguity, one can assume that $\theta(Z)$ 
is the first point where the minimum is achieved.

A standard Brownian motion $W$ on the interval $[t_0,t_1]$, starting at $(t_0,a)$, $a\in\Reals$, is a Markov 
stochastic process with continuous paths, defined 
by the transition probability 
\begin{equation*}
p_{s,r}(x,y)= \frac{1}{\sqrt{2\pi(r-s)}}\exp\Big\{\frac{-(y-x)^2}{2(r-s)}\Big\}
\ ,\quad
t_0\le s<r\le t_1
\ ,
\end{equation*}
and such that $W_{t_0}=a$ with probability 1. 

A Brownian bridge $B$ starting at $(t_0,a)$ and ending at $(t_1,b)$ has the law of a Brownian motion 
defined on the time interval $[t_0,t_1]$ starting at $(t_0,a)$ and conditioned to take the value $b$ at $t_1$. 
The random variable $B_t$, $t_0<t<t_1$, is Gaussian with mean
 $a+\frac{t-t_0}{t_1-t_0}(b-a)$ and variance $\frac{(t-t_0)(t_1-t)}{t_1-t_0}$ .
 
The following results are known or easily deduced (see e.g. Karatzas and Shreve \cite[Sec. 2.8]{MR1121940}):
\begin{prop}
Let $W$ be a Brownian motion starting at $(t_0,a)$, 
defined on the interval $[t_0,t_1]$. The density function of its minimum $m(W)$ is given
by
\begin{equation}\label{eq-fBM}
  f_{m(W)}(y)=\sqrt{\frac{2}{\pi (t_1-t_0)}}
  \exp\Big\{\frac{-(a-y)^2}{2(t_1-t_0)}\Big\}
  \ind_{\{y<a\}}
  \ .
\end{equation}

Let $B$ be a Brownian bridge from  $(t_0,a)$ to $(t_1,b)$.  
The density function ot its
minimum $m(B)$ is given by
\begin{equation}\label{eq-fbridge}
  f_{m(B)}(y)=\frac{2}{t_1-t_0}(a+b-2y)
  \exp\Big\{\frac{-2(a-y)(b-y)}{t_1-t_0}\Big\}\ind_{\{y<a,\, y<b\}}
  \ .
\end{equation}
\qed
\end{prop}

Given $0=t_0<t_1<\cdots<t_n<t_{n+1}=1$, and real values $x_0,\dots,x_{n+1}$, we are interested in 
a stochastic process $X:=\{X_t,\ \in [0,1]\}$ 
whose law is that of a Brownian bridge starting
at $(t_0,x_0)$, ending at $(t_{n+1},x_{n+1})$, and conditioned to pass through all the intermediate points 
$(t_i,x_i)$, $i=1,\dots,n$. 

This process can be thought as the concatenation of $n+1$ independent Brownian bridges 
$B^i:= \{B^i_t,\ t\in[t_i,t_{i+1}]\}$, with respective end values $x_i$ and $x_{i+1}$.
In the optimisation application that we have in mind, the interior points $t_1,\dots,t_n$
are the points sampled by the algorithm, and $x_1,\dots,x_n$ are the observed values at those points.

The law of the minimum of the process $X$ can be expressed in terms of the law of the
minimum of its pieces, in the usual way. Despite the mutual independence of the 
Brownian bridges, this cannot be simplified further:
\begin{prop}
  Let $X$ be the conditioned Brownian bridge defined above,
  and $m(X)$ its minimum. Then, for all $y\in\R$, 
 \begin{equation}\label{eq-1-FmX}
   P\big\{ m(X)>y\big\} =
\prod_{i=0}^n 
\Big(1-\exp\Big\{\frac{-2(x_{i+1}-y)(x_{i}-y)}{t_{i+1}-t_{i}}\Big\}\Big)
 \ind_{\{y<\min(x_{0},\dots,x_{n+1})\}}
 \ .
 \end{equation}

\end{prop}
\begin{proof}
  The formula comes from the standard computation of the law of the minimum of several independent 
  random variables:
 \begin{equation*}
   F_{m(X)}(y) = 1- \prod_{i=0}^n (1-F_{m(B^i)}(y))
   \ ,
 \end{equation*}
  where $F_{m(X)}$ is the distribution function of $m(X)$, and 
  $F_{m(B^i)}$ is the distribution function of the minimum of the Brownian bridge $B^i$, 
  whose density is given by (\ref{eq-fbridge}), adjusting the 
  appropriate constants. 
\end{proof}
  Note that in the case when we do not condition to the end point $(t_{n+1},x_{n+1})$, we obtain a 
  similar expression where, according to (\ref{eq-fBM}), the $n$-th factor in (\ref{eq-1-FmX}) is 
  replaced by 
 \begin{equation*}
  1-\int_{-\infty}^y \sqrt{\frac{2}{\pi (1-t_n)}}
  \exp\Big\{\frac{-(x_n-z)^2}{2(1-t_n)}\Big\}\, dz
  \ .
 \end{equation*}
It would not be difficult to deal with this situation separately
(a conditioned Brownian motion),  
but we will keep our assumptions for simplicity. Moreover, sampling at $t=1$ reverts to our case.

It is natural to try to compute explicitly the density $f_{m(X)}$ of the minimum of $X$ by 
conditioning to each of the intervals $[t_i,t_{i+1}]$:
\begin{equation*}
  f_{m(X)}(y)
  =
  \sum_{i=0}^n
  P\{\theta(X)\in[t_i,t_{i+1}]\}
  \cdot 
  f_{m(X)_{|_{\theta(X)\in[t_i,t_{i+1}]}}}(y)
  \ .
\end{equation*}
  Even though, as we will see, the probability of $\theta(X)$ lying in a given interval can
  be, in principle, computed exactly, the conditional densities in the second factors still depend on the 
  process outside $[t_i,t_{i+1}]$; thus they are not simply densities of the minimum of a 
  single Brownian bridge.

\section{Probability that ${\theta(X)}$ belong to $[t_i,t_{i+1}]$}\label{sec:Pi}
\subsection{Analytical formulae}

The probability that the minimum of $X$ is achieved in one of the intervals $[t_i,t_{i+1}]$ can
be computed exactly:
\begin{prop}
  The probability that the minimum of the process $X$ is located in the interval $[t_i, t_{i+1}]$
  is given by:
 \begin{equation} 
 \begin{split}\label{PA_i}
   P\big\{\theta(X)\in [t_i,t_{i+1}]\big\}=
   \int_{-\infty}^{\min(x_0,\dots,x_{n+1})}
   &\frac{2}{t_{i+1}-t_{i}}(x_i+x_{i+1}-2y)
   \exp\Big\{\frac{-2(x_i-y)(x_{i+1}-y)}{t_{i+1}-t_i}\Big\}
   \times
   \\
   &\prod_{j\neq i} 
  \Big(1-\exp\Big\{\frac{-2(x_{j}-y)(x_{j+1}-y)}{t_{j+1}-t_{j}}\Big\}\Big)
  \, dy
  \ .
 \end{split}
 \end{equation}

\end{prop}
\begin{proof} 
  The random variables $m(B^0),\dots, m(B^n)$ are independent, because of
  the Markov property of Brownian motion. Therefore, their joint density is given by
  the product $\prod_{i=0}^n f_{m(B^i)}(y_i)$, where $f_{m(B^i)}$ is the density of the minimum of the
  $i$-th bridge.
  Denoting, for simplicity, $f_i:=f_{m(B^i)}$ and $F_i$ the corresponding distribution function,
\begin{align*}
  P\big\{\theta(X)\in [t_i,t_{i+1}]\big\} 
  &=
  \int_{\prod\limits_{\stackrel{j=0}{j\neq i}}^n \{y_i<y_j \} }
  f_i(y_i)\times \prod\limits_{\stackrel{j=0}{j\neq i}}^n f_j(y_j) \, dy_0\cdots dy_n
  \\ &=
  \int_{-\infty}^{\infty}   f_i(y_i) \times
  \Big({\prod\limits_{\stackrel{j=0}{j\neq i}}^n}
  \int_{y_i}^{\infty} f_j(y_j)\, dy_j
  \Big)
  \, dy_i
  = 
  \int_{-\infty}^{\infty}   f_i(y) 
  \times{\prod\limits_{\stackrel{j=0}{j\neq i}}^n}
  \big(1-F_j(y)\big)
  \, dy 
  \ .
\end{align*}
  Now, the result 
  is obtained using the densities (\ref{eq-fbridge}) and their distribution functions.
\end{proof}

The integral in (\ref{PA_i}) can be obtained analytically using a computer algebra system.
It is a long expression that
we will not copy here. Let us compare, instead, the minimum on two different intervals:

Let $t_1<t_2\le t_3<t_4$ and consider the Brownian bridge $B_1$ from $(t_1,x_1)$ to 
$(t_2,x_2)$ and the Brownian bridge $B_2$ from $(t_3,x_3)$ to $(t_4,x_4)$. 
Denote $\ell_1:=t_2-t_1$, $d_1:=|x_2-x_1|$, $\ell_2:=t_4-t_3$, $d_2:=|x_4-x_3|$, and 
$\xi:=x_3\wedge x_4-x_1\wedge x_2$. See Figure \ref{fig:two_int}.

\begin{figure}
\begin{center}
\input{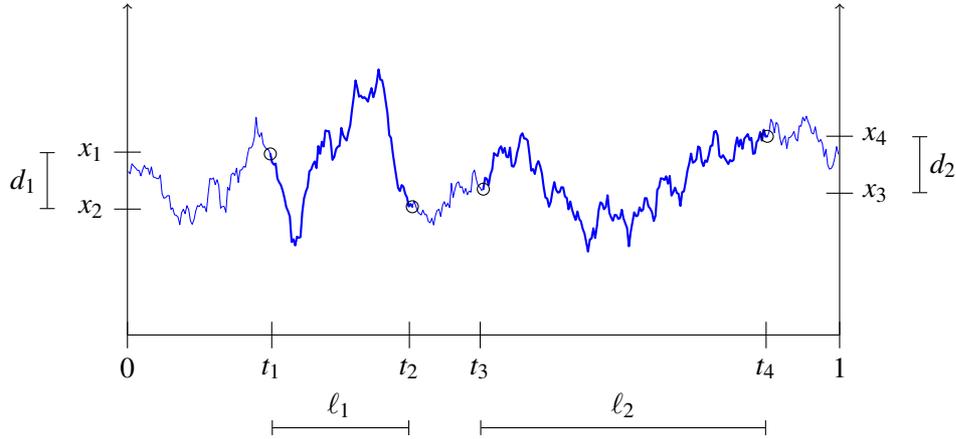}  %\begin{tikzpicture} inside

% \draw [help lines, pink] (0,0) grid (400, 300);
  %AXES:
   %\draw [<->] (59.5,250) -- (59.5,125) -- (350,125);
   \draw [<-] ( 59.5, 250) -- ( 59.5, 125);
   \draw [-]  ( 59.5, 125) -- ( 326, 125);
   \draw [<-] ( 326, 250) -- ( 326, 125);  
  %TICK MARKS
   \draw [-] ( 113.5, 120) node [below] {$t_1$} -- ( 113.5, 130); % +15 -1 -0.5
   \draw [-] ( 165, 120) node [below] {$t_2$} -- ( 165, 130); % +15 
   \draw [-] ( 191.5, 120) node [below] {$t_3$} -- ( 191.5, 130); % -10 +1 +0.5
   \draw [-] ( 298.5, 120) node [below] {$t_4$} -- ( 298.5, 130); % +50 -1 -0.5
   \draw [-] ( 59.5, 120) node [below] {$0$} -- ( 59.5, 130);
   \draw [-] ( 326, 120) node [below] {$1$} -- ( 326, 130);
   
   \draw [-] ( 54.5, 172.5) node [left] {$x_2$} -- ( 64.5, 172.5);
   \draw [-] ( 54.5, 194) node [left] {$x_1$} -- ( 64.5, 194);
   \draw [-] ( 321, 178.5) -- ( 331, 178.5)  node [right] {$x_3$};
   \draw [-] ( 321, 200) -- ( 331, 200) node [right] {$x_4$};
  %BOUNDS   
   \draw [|-|] ( 113.5, 90) -- ( 138.75, 90) node [above] {$\ell_1$} -- ( 165, 90);
   \draw [|-|] ( 191.5, 90) -- ( 245, 90) node [above] {$\ell_2$} --( 298.5, 90);
   \draw [|-|] ( 29.5, 172.5) -- ( 29.5, 183.25) node [left] {$d_1$} -- ( 29.5, 194);
   \draw [|-|] ( 356, 178.5) -- ( 356, 189.25) node [right] {$d_2$} -- ( 356, 200);
\end{tikzpicture}
\end{center}
\caption{A path of Brownian motion conditioned to the circled points. Which of the intervals $[t_1,t_2]$ and $[t_3,t_4]$ 
  is more likely to attain the lowest value?}\label{fig:two_int}
\end{figure}

We ask ourselves which of the two intervals $[t_1,t_2]$ and $[t_3,t_4]$ is more likely to 
contain the lowest value. We have
\begin{align*}
  P\{m(B_1)<m(B_2)\}
  &=
  \int_{\{y<\bar y\}} f_{m(B_1)}(y) f_{m(B_2)}(\bar y)\,dy\,d\bar y
  \\ &=
  \int_{-\infty}^{\infty} f_{m(B_1)}(y) \Big(\int_y^\infty f_{m(B_2)}(\bar y)\,d\bar y\Big)\,dy
  \ .
\end{align*}
Taking as new variables $y-x_0\wedge x_1$ instead of $y$, and $\bar y-x_2\wedge x_3$ instead of $\bar y$, 
we get 
\begin{align*}
  &\int_{-\infty}^{\xi\wedge 0} \frac{2}{\ell_1}(d_1-2y)\exp\Big\{\frac{2y(d_1-y)}{\ell_1}\Big\}
  \Big( 
  \int_{y-\xi}^0 \frac{2}{\ell_2}(d_2-2\bar y)\exp\Big\{\frac{2\bar y(d_2-\bar y)}{\ell_2}\Big\}\,d\bar y\Big)\,dy
  \ ,
\end{align*}
  which can be written
\begin{align}\label{eq:pA1}  
  &
  \int_{-\infty}^{\xi\wedge 0} \frac{2}{\ell_1}(d_1-2y)\exp\Big\{\frac{2y(d_1-y)}{\ell_1}\Big\}
  \Big( 
  1-\exp\Big\{\frac{2(y-\xi)(d_2-(y-\xi))}{\ell_2}\Big\}\Big)\,dy
  \ .
\end{align}
  This integral is also computable analytically.
  Its value depends on five parameters $(\ell_1,d_1,\ell_2,d_2,\xi)$, which are independent from
  each other in a general setting. Therefore, there is no easy way to tell if it is more likely 
  to find the minimum in one interval
  or the other. One observes, as the intuition suggests, that the above probability is
  an increasing function of $\ell_1$, $d_2$ and $\xi$, and that is decreasing in $\ell_2$ and $d_1$, 
  when all the other parameters are fixed.
    
In the case when the intervals are $[0,t_1]$ and $[t_1,1]$, then $\ell_2=1-\ell_1$, and 
$\xi$ can be expressed in terms of $d_1$ and $d_2$, in different ways according to the relative 
positions $x_0<x_1<x_2$, $x_0<x_2<x_1$, or $x_1<x_0\wedge x_2$, so that the number of free parameters reduces 
to three.  

\begin{example}
  Let $B_1$ be the bridge from $(0,0)$ to $(0.5,0)$, and $B_2$ the bridge from $(0.5,0)$ to $(1,d_2)$, 
  for $d_2\ge 0$, and set $p:=P\{m(B_1)<m(B_2)\}$. 
  The following table illustrates how $p$ and $d_2$ are 
  related.
  
  \extrarowheight=1pt
\begin{center}
  \begin{tabular}{|l|c|}
    \hline
    \hfil $p$\hfil  & $d_2$  \\
    \hline
    $0.5$ & $0.0000$ \\ 
    $0.6$ & $0.1837$ \\ 
    $0.7$ & $0.4386$ \\ 
    $0.8$ & $0.8384$ \\ 
    $0.9$ & $1.6620$ \\ 
    $0.95$ & $2.7302$ \\ 
    $0.99$ & $6.8638$ \\ 
    \hline
  \end{tabular} 
\end{center}  
In fact, the explicit functional relationship is given by 
$p=\frac{1}{2}+\sqrt{\pi/8}d_2\exp\{d_2^2/2\}\big(1-\erf\{d_2/\sqrt{2}\}\big)$,
where $\erf()$ is the standard error function. 
If we keep the same first bridge, and make the second shorter and ending at zero, say from 
$(1-\ell_2, 0)$ to $(1,0)$, the dependence between $p$ and the length $\ell_2$ is 
even easier: $p=1/(2\ell_2+1)$. Both are straightforward computations from expression 
(\ref{eq:pA1}). 

By equating both expressions one obtains the variations in $d_2$ and $\ell_2$
that give an equivalent raise of the probability that the first interval contain the 
lowest value.
\end{example}

\subsection{Approximate computation}\label{sec:approxPi}
 
Despite the fact that the integrals (\ref{PA_i}) can be computed analytically, 
the time needed to solve them grows exponentially in the number $n$ of intervals. Indeed,
 the exact computation involves decomposing the integrand in the sum of $O(2^n)$ terms. 
 Each term has an elementary primitive, 
 but 
in an optimisation procedure in which more and more points are sampled,
and consequently the Brownian bridge is conditioned to one more point each time, the computation
becomes cumbersome very quickly.
For example, with just 8 intervals, the computer algebra system \verb|maxima| takes more than
three hours to obtain the result, in an Intel i7 CPU with plenty of memory at its disposal (although
\verb|maxima| only uses one of its cores).  
 It is therefore 
 justified to resort to an approximate method.

 We remark that adding one more point to the set of conditioning
points (that means, splitting one of the intervals in two), forces to recompute from scratch the
probabilities of all intervals. There seems to be no way to reuse previous computations.

As we have seen, the probabilities $P\{m(B^i)<m(B^j)\}$, for each pair of indices $i, j$, 
can be computed exactly and more easily than (\ref{PA_i});
nevertheless, they are not useful even 
to find the interval with the maximal probability. 
An interval $[t_i,t_{i+1}]$ may satisfy
$P\{m(B^i)<m(B^j)\}>1/2$, $\forall j\neq i$, and still not be the interval with the largest
probability of containing $m(X)$. For instance, if we condition the Brownian motion to 
pass through the points
\begin{equation*}
(0, 0),\ (0.144, 0.225),\ (0.610,0.344),\ (1,0.145) \ ,
\end{equation*}
we find that $P\{m(B^1)<m(B^2)\}= 0.5436$ and $P\{m(B^1)<m(B^3)\}= 0.5198$. However, the first interval
is the least
probable one to contain the minimum:
\begin{equation*}
P_{\theta(X)}([t_0,t_1])=0.3124\ ,
\quad
P_{\theta(X)}([t_1,t_2])=0.3374\ ,
\quad
P_{\theta(X)}([t_2,t_3])=0.3502\ .
\end{equation*}
Even more, such a ``winning'' interval may not exist. For instance, conditioning to 
\begin{equation*}
(0,0),\ (0.392,0.031),\ (0.594,-0.157),\ (1,0.435)\ , 
\end{equation*}
one gets the circular relation
$P\{m(B^1)<m(B^2)\}=0.5018$, $P\{m(B^2)<m(B^3)\}=0.5032$, $P\{m(B^3)<m(B^1)\}=0.5013$.

All these arguments support the need to compute (\ref{PA_i}) numerically. 
It is easy to do it with a rigorous error bound: 
For some $\hat x<\min(x_0,\dots,x_{n+1})$, split the integral into the two intervals
$(-\infty, \hat x]$ and $[\hat x,\min\{x_0,\dots,x_{n+1}\}]$. On the first one, the integral 
is bounded by
\begin{align*}
  &\int_{-\infty}^{\hat x} 
  \frac{2}{t_{i+1}-t_{i}}(x_i+x_{i+1}-2y)
   \exp\Big\{\frac{-2(x_i-y)(x_{i+1}-y)}{t_{i+1}-t_i}\Big\}
   \,dy =
  \\ & 
   \exp\Big\{\frac{-2(x_i-\hat x)(x_{i+1}-\hat x)}{t_{i+1}-t_i}\Big\}
   \le \exp\Big\{\frac{-2(x_i\wedge x_{i+1}-\hat x)^2)}{t_{i+1}-t_i}\Big\}
   \ .
\end{align*}
To make this quantity less than a fixed small $\varepsilon$, we can take 
$\hat x < x_i\wedge x_{i+1}-\big(\frac{t_{i+1}-t_i}{2}\log\frac{1}{\varepsilon}\big)^{1/2}$.

For the second interval, denoting the integrand by $f$ and using for instance the standard rectangle 
rule with step size
$h$, the error is bounded by $\frac{1}{2} \|f'\|_\infty\cdot h\cdot L$, where 
$L:=\min(x_0,\dots,x_{n+1})-\hat x $.

Differentiating $f$ and taking into account that all the exponentials take values 
less than 1, one obtains $\|f'\|_\infty \le C$ with   
\begin{equation*}
  C := \frac{4}{t_{i+1}-t_i}
  \Big[1+(x_i+x_{i+1}-2\hat x)\sum_{j=0}^{n} \frac{1}{t_{j+1}-t_j}(x_j+x_{j+1}-2\hat x)\Big]
  \ ,
\end{equation*}
 and the integration step size to ensure an error less than $\varepsilon$ must be
\begin{equation*}
  h \le \frac{2\varepsilon}{C\cdot L}
  \ .
\end{equation*}

A much more efficient method but with a not completely rigorous error bound is given by the 
\verb!quadpack! functions present in the C
Gnu Scientific Library and the Fortran SLATEC Library, which apply a Gauss-Kronrod rule \cite{MR712135}.
With $n=50$, the computation is completed
in less than one-tenth of second, in an Intel i7 CPU at 2.40GHz with 20GB RAM, using 
the  \verb!quadpack! routines implemented in the 
computer algebra system \verb!maxima!, with 
an estimated absolute error rarely bigger than $10^{-9}$. 

The integral of (\ref{PA_i}) can also be transformed into an integral on $[0,1]$ setting 
$y=\min_i x_i-(1-x)/x$ (this is in fact what \verb|quadpack| does), and the new integrand 
does not present any singularity.

\begin{example}\label{ex:approxP1}
  In Table \ref{comparacio_pi}, we show the effective computation of the probability that the minimum
  fall in the first interval, in several situations and with different methods. 
    Sets 1 and 2 comprise four intervals, with end-points at $t=(0,.1,.2,.5,1)$, and values $x=(0,0,0,0,0)$ and $x=(0,.1,.2,.3,.4)$ 
    respectively. Sets 3 and 4 comprise sixteen intervals, with end-points 
    \begin{equation*}
    t=(0,.025,.050,.075,.100,.125,.150,.175,.200,.275,.350,.425,.500,
    .625,.750,.875,1)\ ,
    \end{equation*}
    and all images set to zero in set 3 and to $x=i/40,\ i=0,\dots,16$,
    in set 4.
    
  The methods are:
  1) the analytical computation of the integral (\ref{PA_i}), only in the case of fewest intervals
  (``exact''); 2) the \verb!quadpack! functions through \verb|maxima|; 3) the \verb!romberg! 
  routine built-in in 
  \verb!maxima!; 4) the Riemann approximations with 10\,000 subintervals, taking
  always their left points; 5) the Riemann approximations with the same number of subintervals,
  taking a random point in each one; and 6) the simulation method explained in the next section.   
  In 3),4),5), the computations are also made after the mentioned explicit 
  transformation to the interval [0,1].   
  In 6) a sample of size 10\,000 is taken. For the methods including randomness, 5) and 6), we show
  the highest error observed after 20 realizations.

  All computations were programmed in \verb|maxima|. Time and memory are relative to the fastest 
  and the more economic 
  method in each case; we used the figures reported by \verb|maxima|
  itself in a single run. They give therefore just a rough idea of the computational cost. 
  In the case of 16 intervals, the ``exact'' computation is infeasible and we have taken
  the result of \verb|quadpack| as the base for the figures of the other methods. 
  \begin{table}
  \centering 
{\footnotesize
 \extrarowheight=2pt  
\begin{tabular}{l|lll|lll}
  \hline\hline                                   
  & \multicolumn{3}{l|}{Set 1. Result: 0.05722062072176488 } &\multicolumn{3}{l}{Set 2. Result: 0.3539550244743264 }\\
  \cline{2-7}%\hline
  \phantom{Method}     & error                  & time      & memory    & error                 & time     & memory\\
  \hline
  1) exact             &                        & ~~~~1     & ~~~~1.26  &                       & ~~57.2   & 199\\
  2) quadpack          & $<10^{-16}$            & ~~~~1.07  & ~~~~1     & $<10^{-13}$           & ~~~~1    & ~~~~1 \\
  3) Romberg           & $<10^{-11}$            & ~~~~1.07  & ~~~~1.54  & $<10^{-11}$           & ~~~~1.07 & ~~~~1.63 \\
  4) Riemann left      & $<10^{-16}$            & 122       & 141       & $4.146\times 10^{-6}$ & 121      & 137\\
  5) Riemann random    & $1.008\times 10^{-6}$  & 103       & ~81.8     & $4.283\times 10^{-6}$ & 100      & ~79.0\\
  6) simulation        & $4.179\times 10^{-3}$  & 249       & 191       & $6.045\times 10^{-3}$ & 252      & 188\\
  \hline\hline                                   
  & \multicolumn{3}{l|}{Set 3. Result: 0.003053658531871728 } &\multicolumn{3}{l}{Set 4. Result: 0.3498434691309963 }\\
  \cline{2-7}%\hline
  \phantom{Method}     & error                 & time      & memory    &        error                  & time     & memory\\
  \hline
  2) quadpack          &                       & ~~~~1     & ~~~~1     &                        & ~~~~1    & ~~~~1 \\
  3) Romberg           & $<10^{-12}$           & ~~~~3.41  & ~~~~2.89  & $<10^{-11}$            & ~~~~3.52 & ~~~~2.91 \\
  4) Riemann left      & $<10^{-18}$           & 197       & 294       & $<10^{-7}$             & 199      & 158\\
  5) Riemann random    & $1.366\times 10^{-7}$ & 143       & 157       & $8.140\times 10^{-6}$  & 144      & ~84.3\\
  6) simulation        & $1.446\times 10^{-3}$ & 459       & 462       & $1.06\times 10^{-2}$   & 456      & 254\\  %& ~.2594] &  &  \\
  \hline
\end{tabular}

}
  %\end{center}
  \caption{See Example \ref{ex:approxP1}}   \label{comparacio_pi}                      
  \end{table}
  
\end{example}

\section{Simulating the law of the minimum}\label{sec:simmin}
  We are interested in approximating in an effective way the law of the minimum of the 
  Brownian motion conditioned to the points $(t_0,x_0), \dots, (t_{n+1},x_{n+1})$, with $t_0=0$ and $t_{n+1}=1$,
  so 
  that particular parameters such as its moments can also be easily estimated. To this end,
  taking into account the difficulty and length of the analytical computations implied by
  (\ref{eq-1-FmX}), we resort to simulation. 
    
  A minimum value for each bridge from $(t_i,x_i)$ to $(t_{i+1},x_{i+1})$ can be easily
  simulated from its distribution function $F_{m(B_i)}$, %given in (\ref{eq-Fbridge}),
  which is explicitly invertible:
\begin{equation*}
  F_{m(B_i)}^{-1} (z) = 
{\textstyle\frac{1}{2}}\big(x_i+x_{i+1}-\big((x_{i+1}-x_i)^2-2(t_{i+1}-t_i)\log z\big)^{1/2}\big)
\ ,\quad
  z\in(0,1)
  \ .
\end{equation*}
  Since $m(X)=\min\{m(B^0),\dots,m(B^n)\}$, we can simulate a minimum value of $X$ as 
  the minimum of the simulated minima of each bridge. The computational cost is linear
  in $n$. At the same time, the relative frequency with which each interval contributes 
  to the global minimum constitutes another way to approximate the probabilities 
  $P\{\theta(X)\in [t_i,t_{i+1}]\}$ of Section \ref{sec:approxPi}. This is what is done 
  in row 6 
  of Table \ref{comparacio_pi} for the interval $[0,t_1]$.

  For example, with set 1 of Example \ref{ex:approxP1}, and a sample of size 10\,000, we have obtained the following confidence 
  intervals for the probabilities of each interval to host the minimum:
  \begin{center}{\small
    \extrarowheight=2pt  
    \begin{tabular}{|c|c|}
      \hline
      \hfil interval\hfil  &  95\% C.I.  \\
      \hline
      $[0,0.1]$ & $[0.3501, 0.3715]$ \\ 
      $[0.1,0.2]$ & $[0.0955, 0.1169]$ \\ 
      $[0.2,0.5]$ & $[0.2362, 0.2576]$ \\ 
      $[0.5,1]$ & $[0.2758, 0.2972]$ \\ 
      \hline
    \end{tabular} 
  }
  \end{center}    
  The computations have been done in \verb|R| with the \verb|MultinomialCI| package, based on the
  algorithm of Sison and Glaz \cite{MR1325142}.
  
  Figure  \ref{fig:ajust-imatge} shows the result of simulating the minimum of the process
  in the way described above, 
  conditioned to equispaced points with images equal to zero. The figure includes an histogram and 
  an estimation of the
  density using the polynomial splines algorithm described in \cite{MR1463561}, as implemented in 
  the \verb|logspline| package in \verb|R|.
  %programa per generar la figura: ajust-imatge.R, que produeix el ajust-imatge.tex
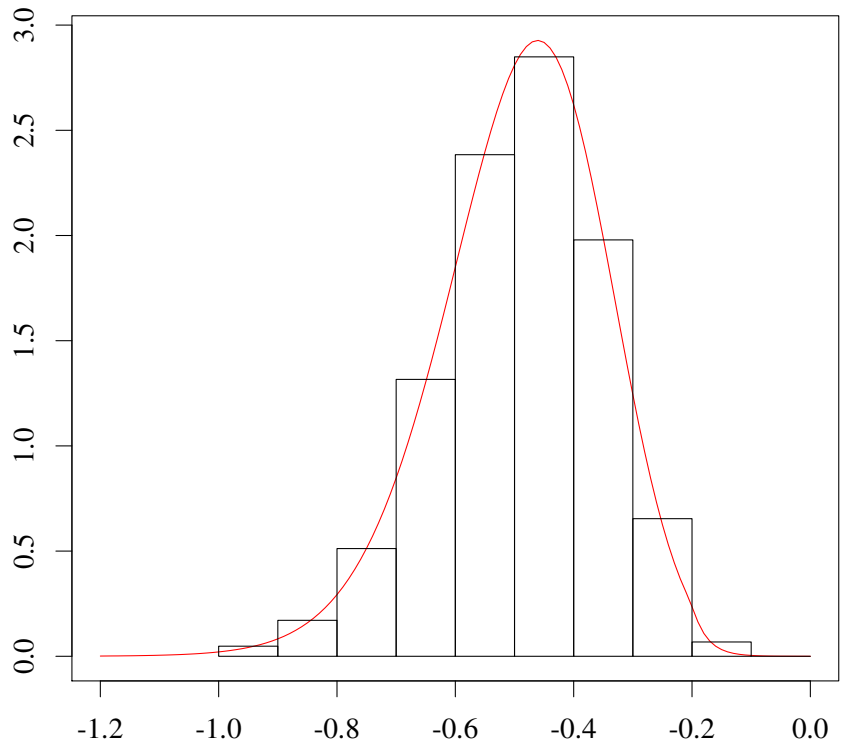
\begin{figure}[th]  
  \begin{center}
    % Created by tikzDevice version 0.8.1 on 2015-05-26 18:09:46
% !TEX encoding = UTF-8 Unicode
\begin{tikzpicture}[x=1pt,y=1pt]
\definecolor{fillColor}{RGB}{255,255,255}
%\path[use as bounding box,fill=fillColor,fill opacity=0.00] (0,0) rectangle (361.35,361.35);
\path[use as bounding box,fill=fillColor,fill opacity=0.00] (50,50) rectangle (311.35,311.35);
\begin{scope}
\path[clip] ( 49.20, 61.20) rectangle (336.15,312.15);
\definecolor{drawColor}{RGB}{255,0,0}

\path[draw=drawColor,line width= 0.4pt,line join=round,line cap=round] ( 59.83, 70.58) --
	( 62.04, 70.60) --
	( 64.26, 70.62) --
	( 66.47, 70.63) --
	( 68.68, 70.66) --
	( 70.90, 70.68) --
	( 73.11, 70.71) --
	( 75.33, 70.75) --
	( 77.54, 70.79) --
	( 79.75, 70.83) --
	( 81.97, 70.89) --
	( 84.18, 70.95) --
	( 86.40, 71.02) --
	( 88.61, 71.10) --
	( 90.83, 71.19) --
	( 93.04, 71.30) --
	( 95.25, 71.42) --
	( 97.47, 71.56) --
	( 99.68, 71.73) --
	(101.90, 71.92) --
	(104.11, 72.14) --
	(106.32, 72.39) --
	(108.54, 72.68) --
	(110.75, 73.01) --
	(112.97, 73.38) --
	(115.18, 73.82) --
	(117.39, 74.31) --
	(119.61, 74.87) --
	(121.82, 75.51) --
	(124.04, 76.24) --
	(126.25, 77.07) --
	(128.47, 78.00) --
	(130.68, 79.06) --
	(132.89, 80.26) --
	(135.11, 81.60) --
	(137.32, 83.11) --
	(139.54, 84.80) --
	(141.75, 86.69) --
	(143.96, 88.80) --
	(146.18, 91.15) --
	(148.39, 93.75) --
	(150.61, 96.63) --
	(152.82, 99.80) --
	(155.03,103.28) --
	(157.25,107.09) --
	(159.46,111.26) --
	(161.68,115.78) --
	(163.89,120.69) --
	(166.11,125.98) --
	(168.32,131.67) --
	(170.53,137.76) --
	(172.75,144.24) --
	(174.96,151.12) --
	(177.18,158.38) --
	(179.39,165.99) --
	(181.60,173.95) --
	(183.82,182.20) --
	(186.03,190.72) --
	(188.25,199.45) --
	(190.46,208.34) --
	(192.67,217.32) --
	(194.89,226.32) --
	(197.10,235.26) --
	(199.32,244.05) --
	(201.53,252.62) --
	(203.75,260.85) --
	(205.96,268.66) --
	(208.17,275.94) --
	(210.39,282.59) --
	(212.60,288.49) --
	(214.82,293.52) --
	(217.03,297.56) --
	(219.24,300.52) --
	(221.46,302.31) --
	(223.67,302.86) --
	(225.89,302.10) --
	(228.10,300.01) --
	(230.32,296.59) --
	(232.53,291.84) --
	(234.74,285.82) --
	(236.96,278.58) --
	(239.17,270.24) --
	(241.39,260.90) --
	(243.60,250.70) --
	(245.81,239.80) --
	(248.03,228.37) --
	(250.24,216.57) --
	(252.46,204.59) --
	(254.67,192.59) --
	(256.88,180.75) --
	(259.10,169.21) --
	(261.31,158.11) --
	(263.53,147.57) --
	(265.74,137.69) --
	(267.96,128.53) --
	(270.17,120.14) --
	(272.38,112.56) --
	(274.60,105.79) --
	(276.81, 99.81) --
	(279.03, 94.59) --
	(281.24, 89.16) --
	(283.45, 83.45) --
	(285.67, 79.12) --
	(287.88, 76.22) --
	(290.10, 74.29) --
	(292.31, 73.02) --
	(294.52, 72.17) --
	(296.74, 71.60) --
	(298.95, 71.23) --
	(301.17, 70.98) --
	(303.38, 70.82) --
	(305.60, 70.71) --
	(307.81, 70.63) --
	(310.02, 70.58) --
	(312.24, 70.55) --
	(314.45, 70.53) --
	(316.67, 70.52) --
	(318.88, 70.51) --
	(321.09, 70.50) --
	(323.31, 70.50) --
	(325.52, 70.49);
\end{scope}
\begin{scope}
\path[clip] (  0.00,  0.00) rectangle (361.35,361.35);
\definecolor{drawColor}{RGB}{0,0,0}

\path[draw=drawColor,line width= 0.4pt,line join=round,line cap=round] ( 59.83, 61.20) -- (325.52, 61.20);

\path[draw=drawColor,line width= 0.4pt,line join=round,line cap=round] ( 59.83, 61.20) -- ( 59.83, 55.20);

\path[draw=drawColor,line width= 0.4pt,line join=round,line cap=round] (104.11, 61.20) -- (104.11, 55.20);

\path[draw=drawColor,line width= 0.4pt,line join=round,line cap=round] (148.39, 61.20) -- (148.39, 55.20);

\path[draw=drawColor,line width= 0.4pt,line join=round,line cap=round] (192.67, 61.20) -- (192.67, 55.20);

\path[draw=drawColor,line width= 0.4pt,line join=round,line cap=round] (236.96, 61.20) -- (236.96, 55.20);

\path[draw=drawColor,line width= 0.4pt,line join=round,line cap=round] (281.24, 61.20) -- (281.24, 55.20);

\path[draw=drawColor,line width= 0.4pt,line join=round,line cap=round] (325.52, 61.20) -- (325.52, 55.20);

\node[text=drawColor,anchor=base,inner sep=0pt, outer sep=0pt, scale=  1.00] at ( 59.83, 39.60) {-1.2};

\node[text=drawColor,anchor=base,inner sep=0pt, outer sep=0pt, scale=  1.00] at (104.11, 39.60) {-1.0};

\node[text=drawColor,anchor=base,inner sep=0pt, outer sep=0pt, scale=  1.00] at (148.39, 39.60) {-0.8};

\node[text=drawColor,anchor=base,inner sep=0pt, outer sep=0pt, scale=  1.00] at (192.67, 39.60) {-0.6};

\node[text=drawColor,anchor=base,inner sep=0pt, outer sep=0pt, scale=  1.00] at (236.96, 39.60) {-0.4};

\node[text=drawColor,anchor=base,inner sep=0pt, outer sep=0pt, scale=  1.00] at (281.24, 39.60) {-0.2};

\node[text=drawColor,anchor=base,inner sep=0pt, outer sep=0pt, scale=  1.00] at (325.52, 39.60) {0.0};

\path[draw=drawColor,line width= 0.4pt,line join=round,line cap=round] ( 49.20, 70.49) -- ( 49.20,308.64);

\path[draw=drawColor,line width= 0.4pt,line join=round,line cap=round] ( 49.20, 70.49) -- ( 43.20, 70.49);

\path[draw=drawColor,line width= 0.4pt,line join=round,line cap=round] ( 49.20,110.18) -- ( 43.20,110.18);

\path[draw=drawColor,line width= 0.4pt,line join=round,line cap=round] ( 49.20,149.87) -- ( 43.20,149.87);

\path[draw=drawColor,line width= 0.4pt,line join=round,line cap=round] ( 49.20,189.56) -- ( 43.20,189.56);

\path[draw=drawColor,line width= 0.4pt,line join=round,line cap=round] ( 49.20,229.25) -- ( 43.20,229.25);

\path[draw=drawColor,line width= 0.4pt,line join=round,line cap=round] ( 49.20,268.94) -- ( 43.20,268.94);

\path[draw=drawColor,line width= 0.4pt,line join=round,line cap=round] ( 49.20,308.64) -- ( 43.20,308.64);

\node[text=drawColor,rotate= 90.00,anchor=base,inner sep=0pt, outer sep=0pt, scale=  1.00] at ( 34.80, 70.49) {0.0};

\node[text=drawColor,rotate= 90.00,anchor=base,inner sep=0pt, outer sep=0pt, scale=  1.00] at ( 34.80,110.18) {0.5};

\node[text=drawColor,rotate= 90.00,anchor=base,inner sep=0pt, outer sep=0pt, scale=  1.00] at ( 34.80,149.87) {1.0};

\node[text=drawColor,rotate= 90.00,anchor=base,inner sep=0pt, outer sep=0pt, scale=  1.00] at ( 34.80,189.56) {1.5};

\node[text=drawColor,rotate= 90.00,anchor=base,inner sep=0pt, outer sep=0pt, scale=  1.00] at ( 34.80,229.25) {2.0};

\node[text=drawColor,rotate= 90.00,anchor=base,inner sep=0pt, outer sep=0pt, scale=  1.00] at ( 34.80,268.94) {2.5};

\node[text=drawColor,rotate= 90.00,anchor=base,inner sep=0pt, outer sep=0pt, scale=  1.00] at ( 34.80,308.64) {3.0};

\path[draw=drawColor,line width= 0.4pt,line join=round,line cap=round] ( 49.20, 61.20) --
	(336.15, 61.20) --
	(336.15,312.15) --
	( 49.20,312.15) --
	( 49.20, 61.20);
\end{scope}
\begin{scope}
\path[clip] ( 49.20, 61.20) rectangle (336.15,312.15);
\definecolor{drawColor}{RGB}{0,0,0}

\path[draw=drawColor,line width= 0.4pt,line join=round,line cap=round] (104.11, 70.49) rectangle (126.25, 74.30);

\path[draw=drawColor,line width= 0.4pt,line join=round,line cap=round] (126.25, 70.49) rectangle (148.39, 84.06);

\path[draw=drawColor,line width= 0.4pt,line join=round,line cap=round] (148.39, 70.49) rectangle (170.53,111.13);

\path[draw=drawColor,line width= 0.4pt,line join=round,line cap=round] (170.53, 70.49) rectangle (192.67,174.96);

\path[draw=drawColor,line width= 0.4pt,line join=round,line cap=round] (192.67, 70.49) rectangle (214.82,259.74);

\path[draw=drawColor,line width= 0.4pt,line join=round,line cap=round] (214.82, 70.49) rectangle (236.96,296.65);

\path[draw=drawColor,line width= 0.4pt,line join=round,line cap=round] (236.96, 70.49) rectangle (259.10,227.59);

\path[draw=drawColor,line width= 0.4pt,line join=round,line cap=round] (259.10, 70.49) rectangle (281.24,122.41);

\path[draw=drawColor,line width= 0.4pt,line join=round,line cap=round] (281.24, 70.49) rectangle (303.38, 75.89);

\path[draw=drawColor,line width= 0.4pt,line join=round,line cap=round] (303.38, 70.49) rectangle (325.52, 70.49);
\end{scope}
\end{tikzpicture}
  \end{center}  
  \caption{Histogram and density estimation of the minimum of a Brownian motion conditioned 
    to pass through the points $\{(k/4,0),\ k=0,\dots,4\}$, with a sample size of 10000
    observations. An asymptotic 95\% symmetric confidence interval for the mean yielded
    $[-0.4939, -0.4884]$. The sample median was -0.4814 .} \label{fig:ajust-imatge} 
\end{figure}

\section{Non-adaptive optimisation}\label{sec:naopt}

  Suppose now that we have a black-box optimisation problem in which we can assume that the Brownian bridge is
  a good probabilistic model for the function at hand. That means, suppose that we are trying to find 
  a point 
  in the interval $[0,1]$ with an image as close as possible to the true minimum of a given but 
  unknown path, 
  drawn at random from the law of a Brownian bridge. The value of the bridge at the end-points is
  assumed to be given, or that they have already been sampled. It is also assumed that we are allowed 
  to sample the path at a fixed small quantity $n$ of points in $[0,1]$.

  A non-adaptive algorithm for this optimisation problem consists in deciding beforehand the $n$ points 
  where we are
  going to sample the path. Any such algorithm has the same convergence order as the best adaptive algorithm as 
  $n\to\infty$, namely $O(n^{-1/2})$, are much simpler to implement,
  and offer parallelisation opportunities.
  Therefore it is worth comparing non-adaptive algorithms in terms of the size of the 
  error incurred for small $n$.    
  In a forthcoming paper we will 
  discuss and compare some adaptive heuristics.
 
\subsection{Two simple strategies}\label{sec:simpEst}
  
  We will first consider and compare two simple and known strategies, 
  whose asymptotic behaviour have already
  been compared (see \cite{MR1325821}), namely: 
  Sampling at equidistant points $\frac{k}{n+1}$, $k=1,\dots,n$, and
  sampling at random uniformly distributed points. 
  We apply both to a bridge with values
  0 and 1 at the end-points and to a symmetric bridge (same value at the end-points). We obtain 
  approximate 95\% confidence intervals for the difference between the minimal sampled value 
  and the true minimum of the path; the results are summarised in Table \ref{comparacio_eqd_rnd}.
  Formally, the confidence intervals estimate the expectation
  \begin{equation*}
  \E\big[\min_{0\le i\le n+1} B_{t_i}-\min_{t\in[0,1]} B_{t}\big]
  \ ,
  \end{equation*}  
  where $B$ is the initial bridge joining $(0,x_0)$ and $(1,x_{n+1})$. In one case the points $t_i$ 
  are fixed; in the other, they are themselves random.
  
  The procedure for the computations is as follows: 
\begin{enumerate}
  \item
  Fix the number of points to sample. We have used $n=2,4,8,16,32,64$ to see the evolution of
  the intervals when the number of points increases. 
  \item\label{it:eqd}
  Sample a path of the Brownian bridge at point $t_1=1/(n+1)$; this is done by simulating a value
  of $B_{t_1}$, which is easy because its law is Gaussian and known. Then, sample at point $t_2=2/(n+1)$ the bridge from $(t_1,x_1)$ to 
  $(1,x_{n+1})$. Proceed similarly to get the values of the path at all equidistant points. 
  \item\label{it:rnd}
  For the simulation at the $n$ random points in $[0,1]$, determine first at which subinterval
  of all previously sampled points the new one belongs to, and sample from the corresponding 
  bridge. The equidistant
  points of step \ref{it:eqd} and their evaluations are included here so that both
   methods are in fact 
  applied to the same path. 
  \item
  From all the $2n$ sampled points of steps \ref{it:eqd} and \ref{it:rnd}, estimate the expectation
  of the minimum of the path to which they belong, with the method described in Section \ref{sec:simmin}. 
  We have used a simulation of size 1000 in this case, 
  taking the mean of the
  values obtained.
  \item\label{it:diff}
  For each sampling strategy, compute the difference between the best sampled point and 
  the estimated minimum of the path.  
  \item
  Repeat steps \ref{it:eqd}--\ref{it:diff} a number of times (we used 1000), and construct
  the asymptotic confidence intervals from the sets of differences obtained, for both strategies.
\end{enumerate}  

  %The results are summarised in Table \ref{comparacio_eqd_rnd} for the two bridges. 
We observe in Table \ref{comparacio_eqd_rnd} that equidistant sampling (`eqd') performs better
than random sampling (`rnd') for both bridges. One also observes that the errors are smaller for the non-symmetric bridge; 
this can be explained by a smaller variance of its minimum value
(these variances can be computed analytically from the density (\ref{eq-fbridge})), despite the fact that
many evaluations are possibly wasted in a non-promising region.  
\begin{table} {\small
  \centering
  \extrarowheight=2pt
  \begin{tabular}{c|c|c|c|c|}
    \cline{2-5}
    & \multicolumn{2}{c|}{Bridge from $(0,0)$ to $(1,1)$} & \multicolumn{2}{c|}{Bridge from $(0,0)$ to $(1,0)$} \\
    \hline
    \multicolumn{1}{|c|}{\# points} & 95\% C.I. eqd & 95\% C.I. rnd & 95\% C.I. eqd & 95\% C.I. rnd \\
    \hline
    \multicolumn{1}{|c|}{$~2$}  & $[0.2390, 0.2517]$ & $[0.2547, 0.2729]$ & $[0.3417, 0.3549]$ & $[0.3791, 0.4012]$ \\
    \multicolumn{1}{|c|}{$~4$}  & $[0.2025, 0.2132]$ & $[0.2163, 0.2320]$ & $[0.2552, 0.2649]$ & $[0.3002, 0.3194]$\\ 
    \multicolumn{1}{|c|}{$~8$}  & $[0.1659, 0.1745]$ & $[0.1759, 0.1891]$ & $[0.1944, 0.2023]$ & $[0.2183, 0.2317]$\\                                                        
    \multicolumn{1}{|c|}{$16$}  & $[0.1280, 0.1341]$ & $[0.1376, 0.1475]$ & $[0.1390, 0.1447]$ & $[0.1651, 0.1760]$\\
    \multicolumn{1}{|c|}{$32$}  & $[0.0920, 0.0963]$ & $[0.1040, 0.1111]$ & $[0.0987, 0.1028]$ & $[0.1198, 0.1283]$\\ 
    \multicolumn{1}{|c|}{$64$}  & $[0.0663, 0.0694]$ & $[0.0778, 0.0838]$ & $[0.0712, 0.0741]$ & $[0.0851, 0.0910]$\\ 
    \hline
  \end{tabular} 
  \caption{Approximate confidence intervals for the expectation of the error when estimating 
    the minimum of a Brownian bridge by the minimum of the sampled values, for equidistant
    (`eqd') and random (`rnd') sampling. The first column is the sample size.} \label{comparacio_eqd_rnd}   
}                     
\end{table}

 Calvin \cite{MR1325821} showed that when $n\to\infty$, the quotient between the errors with 
 equidistant and with random sampling approaches $\approx 0.8239$. Table \ref{quotients_eqd_rnd}
 contains the quotients obtained
 in the simulations. The quotient tends to be bigger when points are few, and it has 
 a decreasing tendency towards the limiting value as the 
 number of points increases.
\begin{table} {\small
  \centering
   \extrarowheight=2pt
  \begin{tabular}{c|c|c|}
    \cline{2-3}
    & {Bridge from $(0,0)$ to $(1,1)$} & {Bridge from $(0,0)$ to $(1,0)$} \\
    \hline
    \multicolumn{1}{|c|}{\# points} & error(eqd) / error(rnd) & error(eqd) / error(rnd) \\
    \hline
    \multicolumn{1}{|c|}{$~2$}  & $0.9301$ & $0.8927$ \\
    \multicolumn{1}{|c|}{$~4$}  & $0.9273$ & $0.8394$ \\ 
    \multicolumn{1}{|c|}{$~8$}  & $0.9326$ & $0.8816$ \\                                                        
    \multicolumn{1}{|c|}{$16$}  & $0.9193$ & $0.8317$ \\
    \multicolumn{1}{|c|}{$32$}  & $0.8754$ & $0.8122$ \\ 
    \multicolumn{1}{|c|}{$64$}  & $0.8397$ & $0.8251$ \\ 
    \hline
  \end{tabular} 
  \caption{Quotient between the errors with equidistant (`eqd') and random (`rnd') sampling.} 
    \label{quotients_eqd_rnd}  
}                        
\end{table}  
 
\subsection{A new algorithm}\label{sec:newAlg}
We propose a new non-adaptive strategy that performs better than equidistant sampling, 
according to the experiments, and that we will call \emph{equiprobable sampling}. 
We sample at the points that divide $[0,1]$ into intervals that
have the same probability to contain the minimum. That means, with the notation of 
Section \ref{sec:Pi}, at points
 $0=t_0<t_1<\cdots<t_n<t_{n+1}=1$ such that
\begin{equation*}
P\{\theta(B)\in [t_i,t_{i+1}]\}=1/(n+1)
\ ,\quad
\text{for all $i=0,\dots,n$\ .}
\end{equation*}
  For a symmetric bridge, these points are those of equidistant sampling. In general, 
  they can be found numerically without much difficulty with the general formula (\ref{eq:fargmin}) 
for the density of the location of the minimum $\theta(B)$ that we are going to prove in Section
\ref{sec:SimLoc}, and that
in this case reduces to
\begin{equation*}
f_{\theta(B)}(s)=
\sqrt{\frac{2(1-s)}{\pi s}} \exp\Big\{\frac{-s}{2(1-s)}\Big\}
\ind_{\{0\le s\le 1\}} 
\ .
\end{equation*} 
Indeed, it is enough to apply a simple bisection method to the distribution function, 
which is increasing, to obtain the unique $t_k$ such that
\begin{equation*}
  \int_{0}^{t_k} f_{\theta(B)}(s) ds = \frac{k}{n+1}\ .
\end{equation*}
Although the integrand is singular at $s=0$, the \verb|quadpack| library \cite{MR712135} 
can handle it
with more than enough precision for our purposes.

\begin{table} {\small
  \centering
  %\begin{center}
  \extrarowheight=2pt
  \begin{tabular}{c|c|c|c|}
    \cline{2-4}
    & \multicolumn{3}{c|} {Bridge from $(0,0)$ to $(1,1)$} \\
    \hline
    \multicolumn{1}{|c|}{\# points} & 95\% C.I. eqp & 95\% C.I. rnd & error(eqp) / error(rnd)\\
    \hline
    \multicolumn{1}{|c|}{$~2$}  & $[0.1975, 0.2144]$ & $[0.2547, 0.2729]$ & 0.7807\\
    \multicolumn{1}{|c|}{$~4$}  & $[0.1637, 0.1780]$ & $[0.2163, 0.2320]$ & 0.7622\\ 
    \multicolumn{1}{|c|}{$~8$}  & $[0.1275, 0.1387]$ & $[0.1759, 0.1891]$ & 0.7293\\                                                        
    \multicolumn{1}{|c|}{$16$}  & $[0.0909, 0.0991]$ & $[0.1376, 0.1475]$ & 0.6664\\
    \multicolumn{1}{|c|}{$32$}  & $[0.0677, 0.0741]$ & $[0.1040, 0.1111]$ & 0.6592\\ 
    \multicolumn{1}{|c|}{$64$}  & $[0.0491, 0.0538]$ & $[0.0778, 0.0838]$ & 0.6368\\ 
    \hline
  \end{tabular} 
  %\end{center}    
  \caption{Equiprobable sampling (`eqp') compared to random sampling (`rnd')
    for a bridge from $(0,0)$ to $(1,1)$. The first column is the sample size and the 
    last one is the quotient of the expected errors.}\label{comparacio_eqp}   
}                      
\end{table}  
With 1000 runs as in the other two methods, we obtained the
results of Table \ref{comparacio_eqp}.
Comparing with Table \ref{comparacio_eqd_rnd}, we see that the expectation
of the error for the equiprobable (`eqp') strategy is the lowest of the three; however, the confidence intervals
for the error are longer than with `eqd', because the variance turns out to be larger. The variances
for `rnd' are the largest. 

Concerning the ratio of expected errors, an asymptotic analysis similar to that of Calvin \cite{MR1325821} is needed
to possibly determine their limit value.

\subsection{Sensitivity analysis}
A natural question that arises here is to what extent the better performance of the equiprobable sampling 
algorithm is tied to the particular law of the Brownian bridge. In other words: what happens if we
sample at points $t_1,\dots,t_n$ computed as in Section \ref{sec:newAlg}, 
according to the Brownian bridge law, when the underlying
process has a different probability law? 
To test this sensitivity issue, we have considered Ornstein--Uhlenbeck bridges with several 
different parameter values.
We recall
first the definition of the Ornstein--Uhlenbeck process and the notion of stochastic bridge in general.

The Ornstein--Uhlenbeck velocity process (O-U, for short), starting 
with value $x_i\in\Reals$ at time $t_i$ is
a Gaussian stochastic process with mean and covariance functions given by
\begin{align*}
  \mu(t) &= x_i\cdot \exp\{-\beta(t-t_i)\}\ ,\quad t\ge t_i\ ,
  \\
  R(t,s) &= \frac{\sigma^2}{2\beta} \big( \exp\{-\beta|t-s|\}-\exp\{-\beta(t+s-2t_i)\}\big)\ ,\quad t,s\ge t_i\ ,
\end{align*} 
  where $\beta>0$ and $\sigma^2>0$. One gets the standard Brownian motion when $\sigma^2=1$ and $\beta\to 0$.

A \emph{bridge} is derived from a given stochastic process $X:=\{X_t,\ t\ge t_i\}$ by conditioning
to a final random variable $X_{t_{i+1}}$. Specifically, for our purposes, an \emph{Ornstein--Uhlenbeck bridge} $B$ from
$(t_i,x_i)$ to $(t_{i+1},x_{i+1})$ is an Ornstein--Uhlenbeck process, starting at $X_{t_i}=x_i$ and conditioned 
to  
$X_{t_{i+1}}=x_{i+1}$. More formally, the law of $B$ coincides with the law of $X$ conditioned to 
the event $\{X_{t_{i+1}}=x_{i+1}\}$.

If $X$ is a Gaussian process, with mean and covariance functions $\mu$ and $R$, it can be shown that
the corresponding bridge $B$ is also Gaussian, with mean and covariance given by
\begin{align*}
\bar\mu(t) &= \mu(t)+\big(x_{i+1}-\mu(t_{i+1})\big)\cdot\frac{R(t,t_{i+1})}{R(t_{i+1},t_{i+1})}\ ,
\quad 
t_i\le t\le t_{i+1}\ ,
\\
\bar R(t,s) &= R(t,s) - \frac{R(t,t_{i+1})R(s,t_{i+1})}{R(t_{i+1},t_{i+1})} \ ,
\quad 
t_i \le t,s\le t_{i+1}\ .
\end{align*} 
See Gasbarra et al.~\cite{MR2397795} for a survey on Gaussian bridges.

We have tested (Brownian bridge)-equiprobable sampling (we will abbreviate it as `Bb-eqp' in the sequel) 
against pure equidistant sampling (`eqd'),
on O-U bridges from $(0,0)$ to $(1,1)$.
The main difficulty here is that the distribution of the minimum of an O-U bridge
is not known. That means we cannot simulate exactly the minimum of each subinterval as 
we did in the case of the Brownian bridge. 

Instead, we approximate the global minimum of 
the path by sampling at $2^{20}$ equidistant points. 
Hence, strictly speaking, what we are doing is a comparison between sampling the O-U path at a few points (up to $2^6$), with
`eqd' an with `Bb-eqp' sampling, versus sampling it at equidistant $2^{20}$ points.
The latter can be done efficiently by 
means of a recursive dyadic partition of $[0,1]$, taking into account that conditioning an O-U bridge
at an interior point results in two independent O-U bridges connected at that point. Therefore, the values at points
of the form $k\cdot 2^{-n}$, $k=1,3,\dots,2^n-1$ can be generated after computing the values at all points of the form 
$k\cdot 2^{-(n-1)}$, and using only these values. 

The resulting computational cost is anyway much bigger than in the simulations
involving only Brownian bridges. For efficiency, in this case they have been coded in \verb|C++| instead of \verb|maxima|, 
using the random number generator of the GNU Scientific Library. 
As before, we have produced 1000 paths for each of the sample sizes $n=2,4,8,16,32,64$. 

For each run, the global minimum of the discretized path is recorded. 
Then, the values at the `eqd' $n$ points and at the `Bb-eqp' 
$n$ points are simulated, taking into account the small bridge of length $2^{-20}$ 
to which they belong. 
Finally, we construct as before 
asymptotic 95\% confidence intervals for the differences between the global minimum of the (discretized)
path and the best sampled point for each method.

The results are summarised in Table \ref{ouEqdEqp}, for the same non-symmetric bridge of Section
\ref{sec:simpEst}. We have considered three different pairs of values for the 
parameters of the O-U bridge:   
\begin{description}
  \item [Case 1]
With $\beta=0.01$, $\sigma^2=1$, 
we have a process with a covariance very similar to that of the Brownian bridge.  
  \item [Case 2]
With  $\beta=4$, $\sigma^2=1$, the variances 
of the O-U bridge variables $B_t$ are always lower than those of the 
Brownian bridge. At the midpoint of the bridge, where the variance is always maximal, 
the O-U bridge variance is approximately half of that of the Brownian bridge. 
  \item [Case 3]
With $\beta=0.01$, $\sigma^2=2$, the variances are always greater than those of the Brownian bridge,
and approximately double at
the midpoint. 
\end{description}

The first thing we notice is that `Bb-eqp' sampling continues to perform better than `eqd' in all cases,
and the results are closer when the variances of the process are small (case 2). 
The confidence intervals are slightly wider for `Bb-eqp', as they were with the Brownian bridge.
 
In case 1, the results are very close to the corresponding ones of 
Tables \ref{comparacio_eqd_rnd} and \ref{comparacio_eqp} (first columns), with slightly 
wider confidence intervals, in both methods. This is to be expected, since the O-U and Brownian bridges are
very similar stochastic processes in this case.
In cases 2 and 3, 
where variances are
sensibly different from Brownian bridge,
we see that the errors are also notably bigger.
Higher variances lead to higher errors (case 3), which is not surprising either.
We can conclude that `Bb-eqp' seems to be better than `eqd' for small samples even when the 
underlying model deviates from the Brownian bridge, at least when it deviates towards the O-U model.

{\renewcommand\tablename{Tables}   
  \begin{table}[ht] {\small
      \centering
      \begin{subtable}{\textwidth}
        \centering
        \extrarowheight=2pt
        \begin{tabular}{c|c|c|c|}
          \cline{2-4}
          & \multicolumn{3}{c|}{95\% C.I. eqd} \\
          \hline
          \multicolumn{1}{|c|}{\# points}      
          & \multicolumn{1}{c|}{$\beta=0.01$, $\sigma^2=1$} 
          & \multicolumn{1}{c|}{$\beta=4$, $\sigma^2=1$} 
          & \multicolumn{1}{c|}{$\beta=0.01$, $\sigma^2=2$}\\
          \hline
          \multicolumn{1}{|c|}{$~2$} & $[0.2454, 0.2675]$ & $[0.2976, 0.3183]$ & $[0.3739, 0.4046]$ \\
          \multicolumn{1}{|c|}{$~4$} & $[0.2039, 0.2215]$ & $[0.2424, 0.2584]$ & $[0.3103, 0.3351]$ \\ 
          \multicolumn{1}{|c|}{$~8$} & $[0.1557, 0.1682]$ & $[0.1834, 0.1962]$ & $[0.2359, 0.2535]$ \\                                                        
          \multicolumn{1}{|c|}{$16$} & $[0.1219, 0.1308]$ & $[0.1324, 0.1417]$ & $[0.1712, 0.1839]$ \\
          \multicolumn{1}{|c|}{$32$} & $[0.0895, 0.0961]$ & $[0.0966, 0.1031]$ & $[0.1331, 0.1423]$ \\ 
          \multicolumn{1}{|c|}{$64$} & $[0.0673, 0.0719]$ & $[0.0709, 0.0756]$ & $[0.0959, 0.1027]$ \\ 
          \hline
        \end{tabular} 
        \caption{Equidistant sampling}
      \end{subtable}      
      \\[2mm]
      \centering
      \begin{subtable}{\textwidth}
        \centering
        \extrarowheight=2pt
        \begin{tabular}{c|c|c|c|}
          \cline{2-4}
          & \multicolumn{3}{c|}{95\% C.I. Bb-eqp} \\
          \hline
          \multicolumn{1}{|c|}{\# points}      
          & \multicolumn{1}{c|}{$\beta=0.01$, $\sigma^2=1$} 
          & \multicolumn{1}{c|}{$\beta=4$, $\sigma^2=1$} 
          & \multicolumn{1}{c|}{$\beta=0.01$, $\sigma^2=2$}\\
          \hline
          \multicolumn{1}{|c|}{$~2$} & $[0.1972, 0.2179]$  & $[0.2858, 0.3087]$ & $[0.3369, 0.3707]$ \\
          \multicolumn{1}{|c|}{$~4$} & $[0.1603, 0.1770]$  & $[0.2235, 0.2428]$ & $[0.2715, 0.2990]$ \\ 
          \multicolumn{1}{|c|}{$~8$} & $[0.1211, 0.1338]$  & $[0.1723, 0.1868]$ & $[0.1980, 0.2184]$ \\                                                        
          \multicolumn{1}{|c|}{$16$} & $[0.0930, 0.1023]$  & $[0.1268, 0.1380]$ & $[0.1562, 0.1733]$ \\
          \multicolumn{1}{|c|}{$32$} & $[0.0681, 0.0749]$  & $[0.0909, 0.0991]$ & $[0.1161, 0.1285]$ \\ 
          \multicolumn{1}{|c|}{$64$} & $[0.0497, 0.0548]$  & $[0.0700, 0.0770]$ & $[0.0855, 0.0950]$ \\ 
          \hline
        \end{tabular} 
        \caption{Equiprobable sampling based on the sampling points of a Brownian bridge}
      \end{subtable}
      \caption{Approximate confidence intervals for the expectation of the error when applying equidistant sampling
        (a) and equiprobable sampling based on the law of a Brownian bridge (b) for the minimum of a 
        Ornstein--Uhlenbeck bridge from $(0,0)$ to $(1,1)$.}\label{ouEqdEqp}                       
    }
  \end{table}    
}

The use of the Ornstein--Uhlenbeck process for testing sensitivity is a natural choice here
since it is the solution of the stochastic differential equation 
\begin{equation*}
\begin{cases}
  dX_t = - \beta X_t +\sigma dW_t\ ,\quad t\ge t_i \\
  X_{t_i}=x_i\ ,
\end{cases}  
\end{equation*}
  where $W$ is a Brownian motion, which models the velocity of a particle under diffusive forces
  with noise intensity $\sigma$, and friction coefficient $\beta$. No friction and unit intensity
  results in a standard Brownian motion.
  
Another natural option for a sensitivity analysis would be the fractional Brownian motion (fBm), that can
be defined as a centred Gaussian process with covariance 
\begin{equation*}
R(t,s)=\frac{1}{2} \big(|t|^{2H}+|s|^{2H}-|t-s|^{2H}\big) \ ,
\end{equation*}
where $0<H<1$. When $H=1/2$ we get again the standard Brownian motion, whereas $H>1/2$ produces  
more regular paths and
positive long-range dependence, and $H<1/2$ leads to more irregular paths and negative
long-range dependence. One can therefore consider
several (or random) values of $H\in[0,1]$ to study the sensitivity of the algorithms.

We have not used fBm as a perturbed model because, besides the fact that the law of the minimum 
of the fB bridge is not known, 
one faces the additional difficulty that conditioning to an interior point does not produce 
two independent bridges, due to the absence of the Markov property. 
Approximating an entire (discretized) path of a fB bridge thus involves a higher cost in time 
and computing memory (see, e.g. Asmussen and Glynn \cite[chapter XI]{MR2331321}). The same
happens with another natural alternative, the integrated Brownian motion, with covariance
\begin{equation*}
R(t,s)=\int_0^t \int_0^s (u\wedge v)\,dv\,du \ ,
\end{equation*}
and whose paths are of class $C^1$ with probability one.   

\section{Simulating the law of the location of the minimum}\label{sec:SimLoc}
The location of the minimum of a continuous function is an ill-posed problem: small
changes in the function may result in big changes in the location of the minimum.
Therefore, the information  about the location of
the minimum given by the sampled values is limited, and possibly of less
practical importance than the information about the minimum value.
Anyway, we can try to visualise this information through the law of 
$\theta(X):=\arg\min_{[0,1]} X_t$.

This law can be simulated with 
the auxiliary use of 
the minima of all bridges $B^0,\dots,B^n$, which in turn can be easily simulated as we have seen 
in Section \ref{sec:simmin}. We prove first that conditioned to all these minima,  
the variables $\theta(X)$ and $\theta(B^j)$,  
where $j$ is the index of the 
interval where the global minimum is attained, have the same law. 
This is the contents of the next proposition:

\begin{prop}\label{le:thx.thbj}
  Denote $x^*:=\min_{0\le i\le n+1} x_i$ and 
  $\Psi_j:=\{y=(y_0,\dots,y_n)\in [-\infty,x^*]^{n+1}\, :\ y_j=\min_i y_i\}$.
  For any Borel set $A\subset[0,1]$, 
\begin{equation*}
  P\big\{\condprob{\theta(X)\in A}{ m(B^0)=y_0,\dots, m(B^n)=y_n}\big\}=
  P\big\{\condprob{\theta(B^j)\in A}{m(B^j)=y_j}\big\}
  \ 
\end{equation*}
  on $\Psi_j$ almost everywhere with respect to Lebesgue measure.
\end{prop}

\begin{proof}
First, we prove the equality
\begin{equation*}
  P\big\{\condprob{\theta(X)\in A}{m(B^0)=y_0,\dots, m(B^n)=y_n}\big\}=
  P\big\{\condprob{\theta(B^j)\in A}{m(B^0)=y_0,\dots, m(B^n)=y_n}\big\}
  \ .
\end{equation*}
  on $\Psi_j$ almost everywhere with respect to Lebesgue measure ($y$-a.e. on $\Psi_j$ for short).

It is clear that the support of $\theta(X)$ and of $\theta(B^j)$ with respect to the conditional law
is the interval $I_j:=[t_j,t_{j+1}]$, $y$-a.e. on $\Psi_j$; therefore, we can 
assume $A\subset I_j$.

Denoting $s^*:=\theta(X)$, since $X_s=B^j_s$, $\forall s\in I_j$, and $s^*\in I_j$ almost surely with respect
to the conditional law, $y$-a.e. on $\Psi_j$, we have 
$m(X)=X_{s^*} = B^j_{s^*} = m(B^j)$ .
This implies $\theta(B^j)=s^*$ almost surely, $y$-a.e. on $\Psi_j$, due to the almost sure uniqueness
of the location of the minimum of a Brownian bridge.

Finally, the equality
\begin{equation*}
P\big\{\condprob{\theta(B^j)\in A}{m(B^1)=y_1,\dots, m(B^n)=y_n}\big\} =
P\big\{\condprob{\theta(B^j)\in A}{m(B^j)=y_j}\big\}
\ .
\end{equation*}
comes from the independence of $B^j$ from all the other bridges. 
\end{proof}

From Proposition \ref{le:thx.thbj}, we see that to simulate the location of the minimum of $X$ it is enough
to simulate the minima of all bridges, select the lowest of them $y_j$, and then simulate the 
location $\theta(B^j)$ of the minimum 
of the bridge $B^j$ conditioned only to $m(B^j)=y_j$. We need first the law of the vector
$(m(B^j),\theta(B^j))$. This is stated in the next proposition. We also give the  
marginal law of $\theta(B^j)$, since we have not been able to find it in the literature in this
generality, even though it will not be used directly in the simulation. 

\begin{prop}
 Writing $\ell:=t_{j+1}-t_j$ and  $d:=|x_{j+1}-x_j|$, the minimum of the 
 bridge $B^j$ from $(t_j,x_j)$ to $(t_{j+1},x_{j+1})$ and its location have the joint density
 \begin{equation}\label{eq:fminargmin}
 f_{(m(B^j),\theta(B^j))}(y,s)=
 \frac{(x_j-y)(x_{j+1}-y)\sqrt{2\ell}}{\sqrt{\pi (s-t_j)^3 (t_{j+1}-s)^3}} 
 \exp\Big\{\frac{d^2}{2\ell}-\frac{(x_j-y)^2}{2(s-t_j)}-\frac{(x_{j+1}-y)^2}{2(t_{j+1}-s)}\Big\}
 \ind_{\{y<x_j,\, y<x_{j+1},\, t_j\le s\le t_{j+1}\}} 
 \ ,
 \end{equation} 
 and the density of the location is  
 \begin{equation}\label{eq:fargmin}
 f_{\theta(B^j)}(s)=
 \Big[
 \frac{d}{\ell^{3/2}}\sqrt{\frac{2}{\pi h(s)}} \exp\Big\{\frac{-d^2}{2\ell}h(s)\Big\}
 +
 \frac{\ell-d^2}{\ell^2}\erfc\Big\{\Big(\frac{d^2}{2\ell}h(s)\Big)^{1/2}\Big\}
 \Big]
  \ind_{\{t_j\le s\le t_{j+1}\}} 
  \ ,
 \end{equation}
 where
 \begin{equation*}
 h(s)=
 \begin{cases}
 (t_{j+1}-s)/(s-t_j)\ ,\quad\text{if $x_{j+1}\le x_{j}$}
 \\
 (s-t_j)/(t_{j+1}-s)\ ,\quad\text{if $x_j\le x_{j+1}$}
 \end{cases}
 \end{equation*}
 and $\erfc(\,)$ is the complementary error function $1-\erf(\,)$ .
\end{prop}

\begin{proof}
  The joint law of $W_t$, with $W$ a standard Brownian motion $W=\{W_s,\ s\ge 0\}$ 
  starting at $W_0=a$, its minimum $m_t$ up to time $t$ and the location $\theta_t$
  of this minimum, is known to have the density
  \begin{equation}\label{joint3}
  \begin{split}
  &P_a\{W_t\in db,\ m_t\in dy,\ \theta_t\in ds\}=
  \\
  &\frac{(a-y)(b-y)}{\pi\sqrt{s^3 (t-s)^3}}
  \exp\Big\{ -\frac{(a-y)^2}{2s}-\frac{(b-y)^2}{2(t-s)}\Big\}
  \ind_{\{y< a,\ y< b,\ 0\le s\le t\}}
  \end{split}
  \end{equation}
  (see Karatzas-Shreve \cite[Prop. 2.8.15]{MR1121940}, or Csáki et al. \cite{MR891709}, where
  it is extended to general diffusions); the formula is usually stated for the maximum, but 
  (\ref{joint3}) is easily deduced by symmetry, taking into account that $-W$ is a Brownian 
  motion with starting value $-a$ at time 0.
  
  Consequently, if $W$ starts instead at time $u<t$, we have
  \begin{equation*}
  \begin{split}
   &P_{(u,a)}\{W_t\in db,\ m_{t}\in dy,\ \theta_{t}\in ds\}=
   \\
   &\frac{(a-y)(b-y)}{\pi\sqrt{(s-u)^3 (t-s)^3}}
   \exp\Big\{ -\frac{(a-y)^2}{2(s-u)}-\frac{(b-y)^2}{2(t-s)}\Big\}
   \ind_{\{y< a,\ y< b,\ u\le s\le t\}}
   \ .
   \end{split}
  \end{equation*}
  
  Conditioning to $\{W_t=b\}$, one finds the joint density of the minimum $m$ and its location
  $\theta$ for a Brownian bridge $B$ joining the points $(u,a)$ and $(t,b)$:
  \begin{equation*}%\label{eq:joint.ys}
  \begin{split}
  &P_{(u,a),(t,b)}\{m\in dy,\ \theta\in ds\}=
  \\
  &\frac{(a-y)(b-y)\sqrt{2(t-u)}}{\sqrt{\pi (s-u)^3 (t-s)^3}}
  \exp\Big\{ \frac{(b-a)^2}{2(t-u)}-\frac{(a-y)^2}{2(s-u)}-\frac{(b-y)^2}{2(t-s)}\Big\}
  \ind_{\{y< a,\ y< b,\ u\le s\le t\}}
  \ ,
  \end{split}
  \end{equation*}
which is equivalent to (\ref{eq:fminargmin}).
  
  Integrating out $y$, we get, if $a<b$,
  \begin{align*}
  P_{(u,a),(t,b)}\{\theta\in ds\}=
  &\Bigg[  
  \frac{b-a}{(t-u)^2}\sqrt{\frac{2(t-u)(t-s)}{\pi(s-u)}}
  \exp\Big\{ \frac{-(b-a)^2(t-s)}{2(t-u)(s-u)}\Big\}
  \\
  &+\frac{(t-u)-(b-a)^2}{(t-u)^2}
  \erfc\Big\{(b-a)\sqrt\frac{t-s}{2(t-u)(s-u)}\Big\}
  \Bigg]
   \ind_{\{u\le s\le t\}}
  \ ,
  \end{align*}
  and, in case $a>b$, 
  \begin{align*}
  P_{(u,a),(t,b)}\{\theta\in ds\}=
  &\Bigg[\frac{a-b}{(t-u)^2}\sqrt{\frac{2(t-u)(s-u)}{\pi(t-s)}}
  \exp\Big\{ \frac{-(b-a)^2(s-u)}{2(t-u)(t-s)}\Big\}
  \\
  &+\frac{(t-u)-(b-a)^2}{(t-u)^2}
  \erfc\Big\{(a-b)\sqrt\frac{s-u}{2(t-u)(t-s)}\Big\}
  \Bigg]
  \ind_{\{u\le s\le t\}}
  \ ,
  \end{align*}
  from where we get (\ref{eq:fargmin}). 
\end{proof}

\begin{prop}
The location of the minimum $\theta(B^j)$ conditioned to $m(B^j)$ has a density of the form
\begin{equation}\label{eq:fargmincond}
  f_{{\theta(B^j)}_{|_{\scriptstyle m(B^j)=y}}} (s)
  =
  C(y)(s-t_j)^{-3/2}(t_{j+1}-s)^{-3/2}
  \exp\Big\{-\frac{A(y)}{2(s-t_j)}-\frac{B(y)}{2(t_{j+1}-s)}\Big\}
  \cdot
  \ind_{\{t_j\le s\le t_{j+1}\}}    
\end{equation}
 for $y<x_j,\, y<x_{j+1}$, where $A$, $B$ and $C$ are positive constants depending only on $y$.
\end{prop}
\begin{proof}
  This is an immediate computation from the joint density (\ref{eq:fminargmin}) and the marginal 
  (\ref{eq-fbridge}), yielding (\ref{eq:fargmincond}) with 
\begin{equation*}
\begin{split}
  &A(y)=(x_j-y)^2\ ,\quad B(y)=(x_{j+1}-y)^2 \ ,
  \\ 
  &C(y)=\frac{(t_{j+1}-t_j)^{3/2}(x_j-y)(x_{j+1}-y)}{\sqrt{2\pi}(x_j+x_{j+1}-2y)}
  \exp\Big\{\frac{(x_{j+1}+x_j-2y)^2}{2(t_{j+1}-t_j)}\Big\}
  \ .
  \end{split}
\end{equation*}
\end{proof}
  Notice that the density (\ref{eq:fargmin}) is not bounded as $h(s)$ goes to zero at one of the end-points
  of the interval $[t_j,t_{j+1}]$. Therefore, it is not easy to sample exactly from it. 
  However, the conditional 
  density (\ref{eq:fargmincond}) given a value $y<\min\{x_j,x_{j+1}\}$ it is bounded, which makes it more
  amenable to the acceptance/rejection method   (see e.g. Asmussen and Glynn \cite{MR2331321}).
  The result is a sample from the joint density of the minimum and its location, from where 
  one obtains a sample of the marginal law of the location. This trick, together with Proposition 
  \ref{le:thx.thbj}, will allow us
  to simulate the location of the minimum of the whole process $X$, conditioned to pass through 
  the given set of points.
  
  To apply acceptance/rejection by comparison with a uniform distribution, the global maximum of the 
  function (\ref{eq:fargmincond}) should be easily calculated or approximated from above. 
  This is indeed the case:      
  There are two obvious minima at the end-points $t_j$ and $t_{j+1}$; 
the remaining extremal points are the roots of the 3-degree polynomial
\begin{equation*}
\textstyle
3C(y)\big[(s-t_j)^2(t_{j+1}-s)-(s-t_j)(t_{j+1}-s)^2\big]
+A(y)(t_{j+1}-s)^2
-B(y)(s-t_j)^2
\ ,
\end{equation*}
as can be
seen by differentiating in $s$ and multiplying by $2(s-t_j)^{7/2}(t_{j+1}-s)^{7/2}$. 
This polynomial may have one or three real roots, corresponding to a
unique maximum, or to two maxima, with a minimum in between. 
In any case, the global maximum can be computed exactly and the acceptance/rejection method
can be
implemented for this density. 

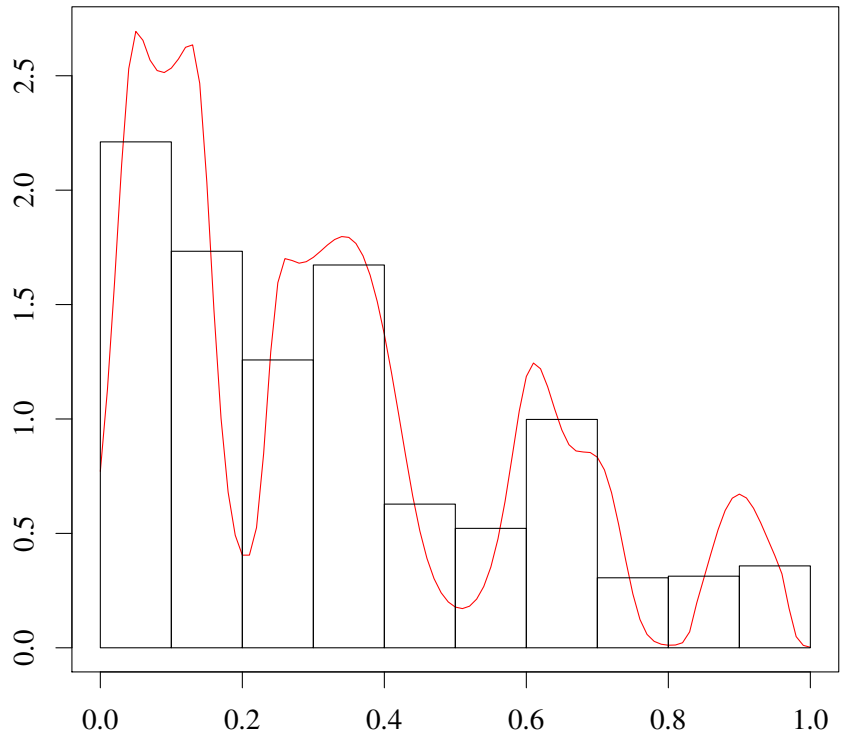
\begin{figure}[t]
  \begin{center}
    % Created by tikzDevice version 0.8.1 on 2015-08-17 19:41:44
% !TEX encoding = UTF-8 Unicode
\begin{tikzpicture}[x=1pt,y=1pt]
\definecolor{fillColor}{RGB}{255,255,255}
%\path[use as bounding box,fill=fillColor,fill opacity=0.00] (0,0) rectangle (361.35,361.35);
\path[use as bounding box,fill=fillColor,fill opacity=0.00] (50,50) rectangle (311.35,311.35);
\begin{scope}
\path[clip] ( 49.20, 61.20) rectangle (336.15,312.15);
\definecolor{drawColor}{RGB}{255,0,0}

\path[draw=drawColor,line width= 0.4pt,line join=round,line cap=round] ( 59.83,136.84) --
	( 62.48,167.61) --
	( 65.14,207.97) --
	( 67.80,252.57) --
	( 70.46,288.78) --
	( 73.11,302.86) --
	( 75.77,299.51) --
	( 78.43,292.03) --
	( 81.08,288.09) --
	( 83.74,287.34) --
	( 86.40,289.02) --
	( 89.05,292.43) --
	( 91.71,296.86) --
	( 94.37,297.78) --
	( 97.02,283.29) --
	( 99.68,246.01) --
	(102.34,197.81) --
	(105.00,156.88) --
	(107.65,129.02) --
	(110.31,112.77) --
	(112.97,105.26) --
	(115.62,105.26) --
	(118.28,115.64) --
	(120.94,143.56) --
	(123.59,181.87) --
	(126.25,208.03) --
	(128.91,217.14) --
	(131.57,216.43) --
	(134.22,215.41) --
	(136.88,215.97) --
	(139.54,217.62) --
	(142.19,219.89) --
	(144.85,222.30) --
	(147.51,224.34) --
	(150.16,225.46) --
	(152.82,225.14) --
	(155.48,222.85) --
	(158.13,218.19) --
	(160.79,210.90) --
	(163.45,200.95) --
	(166.11,188.56) --
	(168.76,174.23) --
	(171.42,158.67) --
	(174.08,142.76) --
	(176.73,127.53) --
	(179.39,114.41) --
	(182.05,104.05) --
	(184.70, 96.41) --
	(187.36, 91.07) --
	(190.02, 87.60) --
	(192.67, 85.66) --
	(195.33, 85.09) --
	(197.99, 86.00) --
	(200.65, 88.67) --
	(203.30, 93.40) --
	(205.96,100.75) --
	(208.62,111.35) --
	(211.27,125.47) --
	(213.93,142.40) --
	(216.59,159.64) --
	(219.24,172.67) --
	(221.90,177.74) --
	(224.56,175.57) --
	(227.22,168.80) --
	(229.87,160.31) --
	(232.53,152.47) --
	(235.19,146.92) --
	(237.84,144.57) --
	(240.50,144.18) --
	(243.16,143.97) --
	(245.81,142.21) --
	(248.47,137.42) --
	(251.13,128.82) --
	(253.78,116.84) --
	(256.44,103.28) --
	(259.10, 90.64) --
	(261.76, 80.99) --
	(264.41, 75.35) --
	(267.07, 72.73) --
	(269.73, 71.64) --
	(272.38, 71.27) --
	(275.04, 71.35) --
	(277.70, 72.21) --
	(280.35, 76.31) --
	(283.01, 87.23) --
	(285.67, 96.53) --
	(288.32,105.86) --
	(290.98,114.82) --
	(293.64,122.11) --
	(296.30,126.75) --
	(298.95,128.28) --
	(301.61,126.82) --
	(304.27,122.95) --
	(306.92,117.51) --
	(309.58,111.31) --
	(312.24,105.07) --
	(314.89, 98.19) --
	(317.55, 85.17) --
	(320.21, 74.52) --
	(322.87, 71.22) --
	(325.52, 70.49);
\end{scope}
\begin{scope}
\path[clip] (  0.00,  0.00) rectangle (361.35,361.35);
\definecolor{drawColor}{RGB}{0,0,0}

\path[draw=drawColor,line width= 0.4pt,line join=round,line cap=round] ( 59.83, 61.20) -- (325.52, 61.20);

\path[draw=drawColor,line width= 0.4pt,line join=round,line cap=round] ( 59.83, 61.20) -- ( 59.83, 55.20);

\path[draw=drawColor,line width= 0.4pt,line join=round,line cap=round] (112.97, 61.20) -- (112.97, 55.20);

\path[draw=drawColor,line width= 0.4pt,line join=round,line cap=round] (166.11, 61.20) -- (166.11, 55.20);

\path[draw=drawColor,line width= 0.4pt,line join=round,line cap=round] (219.24, 61.20) -- (219.24, 55.20);

\path[draw=drawColor,line width= 0.4pt,line join=round,line cap=round] (272.38, 61.20) -- (272.38, 55.20);

\path[draw=drawColor,line width= 0.4pt,line join=round,line cap=round] (325.52, 61.20) -- (325.52, 55.20);

\node[text=drawColor,anchor=base,inner sep=0pt, outer sep=0pt, scale=  1.00] at ( 59.83, 39.60) {0.0};

\node[text=drawColor,anchor=base,inner sep=0pt, outer sep=0pt, scale=  1.00] at (112.97, 39.60) {0.2};

\node[text=drawColor,anchor=base,inner sep=0pt, outer sep=0pt, scale=  1.00] at (166.11, 39.60) {0.4};

\node[text=drawColor,anchor=base,inner sep=0pt, outer sep=0pt, scale=  1.00] at (219.24, 39.60) {0.6};

\node[text=drawColor,anchor=base,inner sep=0pt, outer sep=0pt, scale=  1.00] at (272.38, 39.60) {0.8};

\node[text=drawColor,anchor=base,inner sep=0pt, outer sep=0pt, scale=  1.00] at (325.52, 39.60) {1.0};

\path[draw=drawColor,line width= 0.4pt,line join=round,line cap=round] ( 49.20, 70.30) -- ( 49.20,286.13);

\path[draw=drawColor,line width= 0.4pt,line join=round,line cap=round] ( 49.20, 70.30) -- ( 43.20, 70.30);

\path[draw=drawColor,line width= 0.4pt,line join=round,line cap=round] ( 49.20,113.46) -- ( 43.20,113.46);

\path[draw=drawColor,line width= 0.4pt,line join=round,line cap=round] ( 49.20,156.63) -- ( 43.20,156.63);

\path[draw=drawColor,line width= 0.4pt,line join=round,line cap=round] ( 49.20,199.80) -- ( 43.20,199.80);

\path[draw=drawColor,line width= 0.4pt,line join=round,line cap=round] ( 49.20,242.96) -- ( 43.20,242.96);

\path[draw=drawColor,line width= 0.4pt,line join=round,line cap=round] ( 49.20,286.13) -- ( 43.20,286.13);

\node[text=drawColor,rotate= 90.00,anchor=base,inner sep=0pt, outer sep=0pt, scale=  1.00] at ( 34.80, 70.30) {0.0};

\node[text=drawColor,rotate= 90.00,anchor=base,inner sep=0pt, outer sep=0pt, scale=  1.00] at ( 34.80,113.46) {0.5};

\node[text=drawColor,rotate= 90.00,anchor=base,inner sep=0pt, outer sep=0pt, scale=  1.00] at ( 34.80,156.63) {1.0};

\node[text=drawColor,rotate= 90.00,anchor=base,inner sep=0pt, outer sep=0pt, scale=  1.00] at ( 34.80,199.80) {1.5};

\node[text=drawColor,rotate= 90.00,anchor=base,inner sep=0pt, outer sep=0pt, scale=  1.00] at ( 34.80,242.96) {2.0};

\node[text=drawColor,rotate= 90.00,anchor=base,inner sep=0pt, outer sep=0pt, scale=  1.00] at ( 34.80,286.13) {2.5};

\path[draw=drawColor,line width= 0.4pt,line join=round,line cap=round] ( 49.20, 61.20) --
	(336.15, 61.20) --
	(336.15,312.15) --
	( 49.20,312.15) --
	( 49.20, 61.20);
\end{scope}
\begin{scope}
\path[clip] ( 49.20, 61.20) rectangle (336.15,312.15);
\definecolor{drawColor}{RGB}{0,0,0}

\path[draw=drawColor,line width= 0.4pt,line join=round,line cap=round] ( 59.83, 70.30) rectangle ( 86.40,261.18);

\path[draw=drawColor,line width= 0.4pt,line join=round,line cap=round] ( 86.40, 70.30) rectangle (112.97,219.91);

\path[draw=drawColor,line width= 0.4pt,line join=round,line cap=round] (112.97, 70.30) rectangle (139.54,178.90);

\path[draw=drawColor,line width= 0.4pt,line join=round,line cap=round] (139.54, 70.30) rectangle (166.11,214.73);

\path[draw=drawColor,line width= 0.4pt,line join=round,line cap=round] (166.11, 70.30) rectangle (192.67,124.51);

\path[draw=drawColor,line width= 0.4pt,line join=round,line cap=round] (192.67, 70.30) rectangle (219.24,115.36);

\path[draw=drawColor,line width= 0.4pt,line join=round,line cap=round] (219.24, 70.30) rectangle (245.81,156.46);

\path[draw=drawColor,line width= 0.4pt,line join=round,line cap=round] (245.81, 70.30) rectangle (272.38, 96.71);

\path[draw=drawColor,line width= 0.4pt,line join=round,line cap=round] (272.38, 70.30) rectangle (298.95, 97.32);

\path[draw=drawColor,line width= 0.4pt,line join=round,line cap=round] (298.95, 70.30) rectangle (325.52,101.20);
\end{scope}
\end{tikzpicture}
  \end{center}  
  \caption{Density estimation for the location of the minimum of the Brownian bridge conditioned
    to the points $(0,0), (0.2, 0.06), (0.5, 0.16), (0.8, 0.26), (1, 0.20)$ .}  \label{fig:histograma-arg-1} 
\end{figure}  

In Figure \ref{fig:histograma-arg-1}, we see the result of a simulation of size 10\,000,
with a density estimation using the \verb|logsplines| method in \verb|R|. As it was remarked in Section
\ref{sec:prelim}, once we are considering the minimum of the whole process, the different
bridges are no longer independent; in particular, the shape of the density in each subinterval
is not the one to be expected from formula (\ref{eq:fargmin}), and in fact it is quite difficult
to predict from the conditioning values. Hence the interest to have an 
exact simulation method.

As a more clear example of the last remark, consider the concatenation of two symmetric
bridges, from $(0,0)$ to $(0.5,0)$, and from $(0.5,0)$ to $(1,0)$. Separately, the location of their 
minima follows a uniform distribution; however, the location of the global
minimum follows the density simulated in Figure \ref{fig:histograma-arg-4}. The fact that the minimum of the two
minima tends to take a lower value than a single minimum drifts away its location from the
end-points of the subintervals.

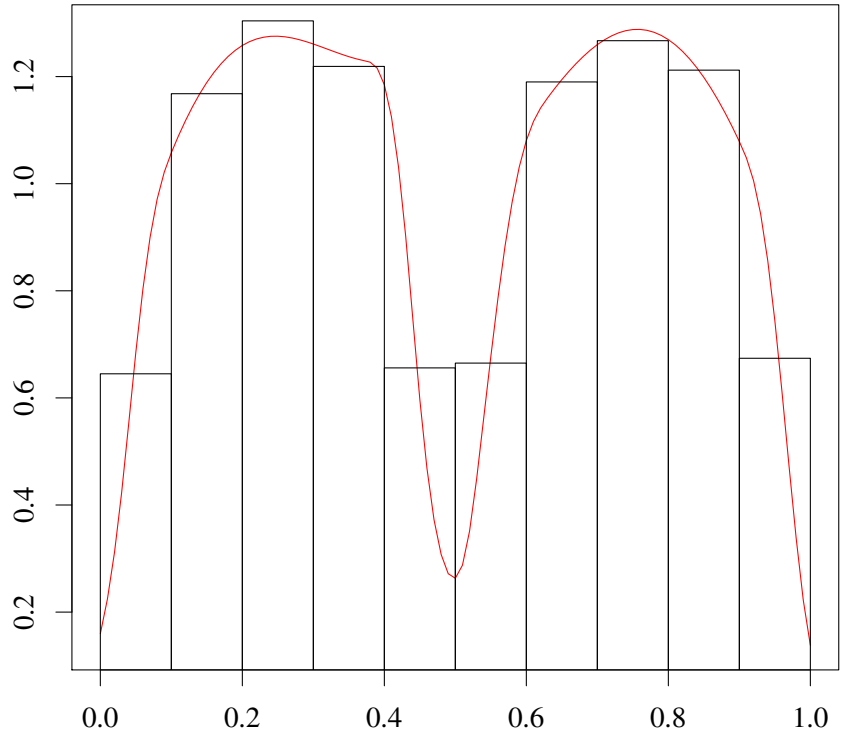
\begin{figure}[t]
  \begin{center}
    % Created by tikzDevice version 0.8.1 on 2015-08-17 20:32:26
% !TEX encoding = UTF-8 Unicode
\begin{tikzpicture}[x=1pt,y=1pt]
\definecolor{fillColor}{RGB}{255,255,255}
%\path[use as bounding box,fill=fillColor,fill opacity=0.00] (0,0) rectangle (361.35,361.35);
\path[use as bounding box,fill=fillColor,fill opacity=0.00] (50,50) rectangle (311.35,311.35);
\begin{scope}
\path[clip] ( 49.20, 61.20) rectangle (336.15,312.15);
\definecolor{drawColor}{RGB}{255,0,0}

\path[draw=drawColor,line width= 0.4pt,line join=round,line cap=round] ( 59.83, 74.85) --
	( 62.48, 87.92) --
	( 65.14,105.60) --
	( 67.80,128.06) --
	( 70.46,154.21) --
	( 73.11,181.02) --
	( 75.77,204.94) --
	( 78.43,224.32) --
	( 81.08,238.77) --
	( 83.74,248.96) --
	( 86.40,256.34) --
	( 89.05,262.65) --
	( 91.71,268.46) --
	( 94.37,273.78) --
	( 97.02,278.58) --
	( 99.68,282.86) --
	(102.34,286.63) --
	(105.00,289.89) --
	(107.65,292.66) --
	(110.31,294.96) --
	(112.97,296.81) --
	(115.62,298.23) --
	(118.28,299.26) --
	(120.94,299.92) --
	(123.59,300.26) --
	(126.25,300.29) --
	(128.91,300.07) --
	(131.57,299.63) --
	(134.22,299.00) --
	(136.88,298.23) --
	(139.54,297.35) --
	(142.19,296.39) --
	(144.85,295.40) --
	(147.51,294.41) --
	(150.16,293.46) --
	(152.82,292.57) --
	(155.48,291.79) --
	(158.13,291.15) --
	(160.79,290.50) --
	(163.45,288.20) --
	(166.11,281.96) --
	(168.76,269.90) --
	(171.42,250.88) --
	(174.08,225.00) --
	(176.73,194.02) --
	(179.39,163.21) --
	(182.05,137.27) --
	(184.70,117.83) --
	(187.36,104.77) --
	(190.02, 97.49) --
	(192.67, 95.79) --
	(195.33,100.68) --
	(197.99,113.52) --
	(200.65,133.28) --
	(203.30,156.30) --
	(205.96,179.46) --
	(208.62,201.43) --
	(211.27,221.10) --
	(213.93,237.71) --
	(216.59,250.95) --
	(219.24,260.93) --
	(221.90,268.11) --
	(224.56,273.19) --
	(227.22,277.07) --
	(229.87,280.59) --
	(232.53,283.92) --
	(235.19,287.05) --
	(237.84,289.95) --
	(240.50,292.61) --
	(243.16,295.01) --
	(245.81,297.12) --
	(248.47,298.93) --
	(251.13,300.42) --
	(253.78,301.58) --
	(256.44,302.38) --
	(259.10,302.81) --
	(261.76,302.86) --
	(264.41,302.51) --
	(267.07,301.76) --
	(269.73,300.59) --
	(272.38,299.00) --
	(275.04,297.00) --
	(277.70,294.56) --
	(280.35,291.71) --
	(283.01,288.44) --
	(285.67,284.75) --
	(288.32,280.66) --
	(290.98,276.18) --
	(293.64,271.32) --
	(296.30,266.10) --
	(298.95,260.54) --
	(301.61,254.34) --
	(304.27,245.80) --
	(306.92,233.24) --
	(309.58,215.61) --
	(312.24,192.91) --
	(314.89,166.31) --
	(317.55,138.05) --
	(320.21,110.92) --
	(322.87, 87.53) --
	(325.52, 70.49);
\end{scope}
\begin{scope}
\path[clip] (  0.00,  0.00) rectangle (361.35,361.35);
\definecolor{drawColor}{RGB}{0,0,0}

\path[draw=drawColor,line width= 0.4pt,line join=round,line cap=round] ( 59.83, 61.20) -- (325.52, 61.20);

\path[draw=drawColor,line width= 0.4pt,line join=round,line cap=round] ( 59.83, 61.20) -- ( 59.83, 55.20);

\path[draw=drawColor,line width= 0.4pt,line join=round,line cap=round] (112.97, 61.20) -- (112.97, 55.20);

\path[draw=drawColor,line width= 0.4pt,line join=round,line cap=round] (166.11, 61.20) -- (166.11, 55.20);

\path[draw=drawColor,line width= 0.4pt,line join=round,line cap=round] (219.24, 61.20) -- (219.24, 55.20);

\path[draw=drawColor,line width= 0.4pt,line join=round,line cap=round] (272.38, 61.20) -- (272.38, 55.20);

\path[draw=drawColor,line width= 0.4pt,line join=round,line cap=round] (325.52, 61.20) -- (325.52, 55.20);

\node[text=drawColor,anchor=base,inner sep=0pt, outer sep=0pt, scale=  1.00] at ( 59.83, 39.60) {0.0};

\node[text=drawColor,anchor=base,inner sep=0pt, outer sep=0pt, scale=  1.00] at (112.97, 39.60) {0.2};

\node[text=drawColor,anchor=base,inner sep=0pt, outer sep=0pt, scale=  1.00] at (166.11, 39.60) {0.4};

\node[text=drawColor,anchor=base,inner sep=0pt, outer sep=0pt, scale=  1.00] at (219.24, 39.60) {0.6};

\node[text=drawColor,anchor=base,inner sep=0pt, outer sep=0pt, scale=  1.00] at (272.38, 39.60) {0.8};

\node[text=drawColor,anchor=base,inner sep=0pt, outer sep=0pt, scale=  1.00] at (325.52, 39.60) {1.0};

\path[draw=drawColor,line width= 0.4pt,line join=round,line cap=round] ( 49.20, 83.01) -- ( 49.20,285.06);

\path[draw=drawColor,line width= 0.4pt,line join=round,line cap=round] ( 49.20, 83.01) -- ( 43.20, 83.01);

\path[draw=drawColor,line width= 0.4pt,line join=round,line cap=round] ( 49.20,123.42) -- ( 43.20,123.42);

\path[draw=drawColor,line width= 0.4pt,line join=round,line cap=round] ( 49.20,163.83) -- ( 43.20,163.83);

\path[draw=drawColor,line width= 0.4pt,line join=round,line cap=round] ( 49.20,204.24) -- ( 43.20,204.24);

\path[draw=drawColor,line width= 0.4pt,line join=round,line cap=round] ( 49.20,244.65) -- ( 43.20,244.65);

\path[draw=drawColor,line width= 0.4pt,line join=round,line cap=round] ( 49.20,285.06) -- ( 43.20,285.06);

\node[text=drawColor,rotate= 90.00,anchor=base,inner sep=0pt, outer sep=0pt, scale=  1.00] at ( 34.80, 83.01) {0.2};

\node[text=drawColor,rotate= 90.00,anchor=base,inner sep=0pt, outer sep=0pt, scale=  1.00] at ( 34.80,123.42) {0.4};

\node[text=drawColor,rotate= 90.00,anchor=base,inner sep=0pt, outer sep=0pt, scale=  1.00] at ( 34.80,163.83) {0.6};

\node[text=drawColor,rotate= 90.00,anchor=base,inner sep=0pt, outer sep=0pt, scale=  1.00] at ( 34.80,204.24) {0.8};

\node[text=drawColor,rotate= 90.00,anchor=base,inner sep=0pt, outer sep=0pt, scale=  1.00] at ( 34.80,244.65) {1.0};

\node[text=drawColor,rotate= 90.00,anchor=base,inner sep=0pt, outer sep=0pt, scale=  1.00] at ( 34.80,285.06) {1.2};

\path[draw=drawColor,line width= 0.4pt,line join=round,line cap=round] ( 49.20, 61.20) --
	(336.15, 61.20) --
	(336.15,312.15) --
	( 49.20,312.15) --
	( 49.20, 61.20);
\end{scope}
\begin{scope}
\path[clip] ( 49.20, 61.20) rectangle (336.15,312.15);
\definecolor{drawColor}{RGB}{0,0,0}

\path[draw=drawColor,line width= 0.4pt,line join=round,line cap=round] ( 59.83, 42.60) rectangle ( 86.40,172.92);

\path[draw=drawColor,line width= 0.4pt,line join=round,line cap=round] ( 86.40, 42.60) rectangle (112.97,278.60);

\path[draw=drawColor,line width= 0.4pt,line join=round,line cap=round] (112.97, 42.60) rectangle (139.54,306.08);

\path[draw=drawColor,line width= 0.4pt,line join=round,line cap=round] (139.54, 42.60) rectangle (166.11,288.90);

\path[draw=drawColor,line width= 0.4pt,line join=round,line cap=round] (166.11, 42.60) rectangle (192.67,175.15);

\path[draw=drawColor,line width= 0.4pt,line join=round,line cap=round] (192.67, 42.60) rectangle (219.24,176.96);

\path[draw=drawColor,line width= 0.4pt,line join=round,line cap=round] (219.24, 42.60) rectangle (245.81,283.04);

\path[draw=drawColor,line width= 0.4pt,line join=round,line cap=round] (245.81, 42.60) rectangle (272.38,298.60);

\path[draw=drawColor,line width= 0.4pt,line join=round,line cap=round] (272.38, 42.60) rectangle (298.95,287.49);

\path[draw=drawColor,line width= 0.4pt,line join=round,line cap=round] (298.95, 42.60) rectangle (325.52,178.78);
\end{scope}
\end{tikzpicture}
  \end{center}  
  \caption{Density estimation for the location of the global minimum of two concatenated 
    symmetric identical Brownian bridges.}  \label{fig:histograma-arg-4} 
\end{figure}  

\section{Acknowledgements}
  This work has been supported by grants numbers MTM2014-59179-C2-1-P  
  from the Ministry of Economy and Competitiveness of Spain; and 
  UNAB10-4E-378, co-funded by the European Regional Development Fund (ERDF). 
  
  We would also like to thank prof.~Jim Pitman for suggestions on the computations
  of Section \ref{sec:SimLoc}.

\bibliographystyle{plain}
%\bibliography{MinCondBM}
%\begin{comment}

%\end{comment}  
\end{document}